\documentclass[10pt, article]{amsart}
\usepackage{tikz}
\usetikzlibrary{calc}
\usepackage{ae} % or {zefonts}
\usepackage[T1]{fontenc}
\usepackage[cp1250]{inputenc}
\usepackage{amsmath}
\usepackage{amssymb, amsfonts,amscd,verbatim}

\usepackage[normalem]{ulem}
\usepackage{hyperref}
\usepackage{indentfirst}
\usepackage{latexsym}
\input xy
\xyoption{all}

\usepackage{xcolor}

\usepackage{amsmath}    % need for subequations
\theoremstyle{plain}
\newtheorem{Pocz}{Poczatek}[section]
\newtheorem{Proposition}[Pocz]{Proposition}

\newtheorem{Theorem}[Pocz]{Theorem}
\newtheorem{Corollary}[Pocz]{Corollary}

\newtheorem{Lemma}[Pocz]{Lemma}
\newtheorem{Observation}[Pocz]{Observation}
\newtheorem{Notation}[Pocz]{Notation}

\newtheorem{Example}[Pocz]{Example}

\theoremstyle{definition}
\newtheorem{Definition}[Pocz]{Definition}

\theoremstyle{remark}

\numberwithin{equation}{section}
\title[Ends of large scale groups]
{Ends of large scale groups}

\author{Yuankui Ma}
\address{Xi'an Technological University, No.2 Xuefu zhong lu, Weiyang district, Xi'an, China 710021}
\email{mayuankui@xatu.edu.cn}

\author{Hussain Rashed}
\address{University of Tennessee, Knoxville, TN 37996, USA}
\email{hrashed1@vols.utk.edu}

\author{Jerzy Dydak}
\address{University of Tennessee, Knoxville, TN 37996, USA}
\email{jdydak@utk.edu}
\address{Xi'an Technological University, No.2 Xuefu zhong lu, Weiyang district, Xi'an, China 710021}
\email{jdydak@gmail.com}

\date{ \today
}


\begin{document}
\maketitle
\begin{center}
\today
\end{center}

\tableofcontents

\begin{abstract}
The aim of this paper is to unify the theory of ends of finitely generated groups with that of ends of locally compact, metrizable and connected topological groups. In both theories one proves that, if the number of ends is finite, then it must be at most $2$. In both theories groups of two ends are characterized as having an infinite cyclic subgroup of either finite index or such that its coset space is compact. Our generalization amounts to defining the space of ends of any coarse space and then applying it to large scale groups, a class of groups generalizing both finitely generated groups and locally compact, metrizable and connected topological groups.

Additionally, we prove a version of Svarc-Milnor Lemma for large scale groups and we prove that coarsely hyperbolic large scale groups have finite asymptotic dimension provided they have bounded geometry.
\end{abstract}

\section{Introduction}

Historically, as noted in \cite{DK} on p.287, 
ends are the oldest coarse topological notion. Here is their internal description:

\begin{Definition}
A \textbf{Freudenthal end} is a decreasing sequence $(U_i)_{i\ge 1}$ of components of $X\setminus K_i$, where $(K_i)_{i\ge 1}$ is an \textbf{exhausting sequence}, i.e. $K_i$ is compact, $K_i\subset int(K_{i+1})$ for each $i\ge 1$, and $\bigcup\limits_{i=1}^\infty K_i=X.$
The set of ends of $X$ is denoted by $Ends(X)$.
\end{Definition}

 Ends were used by Freudenthal in 1930 in his famous compactification (see \cite{Peschke}  for information about theorems in this section and see \cite{Engel} for results related to the theory of dimension):

\begin{Theorem}
Suppose $X$ is a $\sigma$-compact, locally compact, connected and locally connected Hausdorff space. It has a compactification $\bar X$ such that
$\bar X\setminus X$ is of dimension $0$ and $\bar X$ dominates any compactification
$\hat X$ of $X$ whose corona is of dimension $0$.
\end{Theorem}

\begin{Notation}
Given a $\sigma$-compact, locally compact, connected and locally connected Hausdorff space $(X,\mathcal{T})$. If $U\in \mathcal{T}$, we define the subset $U_{end}:=\{(U_i)\in Ends(X)|U_{i}\subset U$ for some $i\geq 1\}$ and $\widetilde{U}_{end}:=U_{end}\cup U$.
\end{Notation}
The family $\mathcal{T}\cup \{\widetilde{U}_{end}| U\in \mathcal{T}\}$ is a basis for a topology $\mathcal{T}_{end}$ on $X\cup Ends(X)$. The topological space $\bar X:=X\cup Ends(X)$ is a compactification of $X$ called the \textbf{Freudenthal compactification}. The space of ends $Ends(X)=\bar X\setminus X$ is of dimension 0, and $\bar X$ dominates any compactification
$\hat X$ of $X$ whose corona is of dimension $0$. Moreover, the number of ends of $X$ is the supremum of $n_i\geq 0$ where $n_i$ is the number of all mutually disjoint unbounded components of $X\setminus K_i$, for all $i\geq 1$.

Initially, ends were useful as properties of topological groups:
\begin{Theorem}
(Freudenthal) A path connected topological group has at most two ends.
\end{Theorem}

\begin{Theorem} (Leo Zippin \cite{Zippin})
If a locally compact, metrizable, connected topological group $G$ is two-ended, then $G$ contains a closed subgroup $T$ isomorphic to the group of reals such that the coset-space $G/T$ is compact; moreover, the space $G$ is the topological product of the axis of reals by a compact connected set homeomorphic to the space $G/T$.
\end{Theorem}

\begin{Theorem}
(H. Hopf) Let $G$ be a finitely generated discrete group acting on a space $X$ by covering transformations. Suppose the orbit space $B:=X/G$ is compact. Then (i) and (ii), below, hold.\\
(i) The end space of $X$ has 0, 1 or 2 (discrete) elements or is a Cantor space.\\ 
(ii) If $G$ also acts on $Y$ satisfying the hypotheses above, then $X$ and $Y$ have
homeomorphic end spaces.
\end{Theorem}

Conclusion (ii) suggests to regard the end space of $X$ as an invariant of the group $G$ itself:
\begin{Definition}
Let $p:X \to B$ be a covering map with compact base $B$ and the group of covering transformations $G$. The \textbf{end space} of $G$ is
$$Ends(G):= Ends(X).$$
\end{Definition}

When applied to a Cayley graph of $G$, it gives the standard definition of ends of finitely generated groups (see \cite{DK}, p.295). See \cite{Geog}  for basic results in this theory and see \cite{Grom} for more general facts in coarse geometry related to groups. \cite{MM} contains interesting results for ends of finitely generated groups.

In this paper we will define ends of arbitrary countable groups by generalizing the construction of the Higson corona. In the case of coarse spaces we generalize Freudenthal's method to define their space of ends.

E. Specker \cite{Sp} defined ends of arbitrary groups using Stone's duality theorem. See a very nice paper \cite{Corn} of Yves Cornulier describing properties of the space of ends of infinitely generated groups. We consider Specker's approach highly non-geometric. Additionally, our way of defining ends of spaces leads directly to view them as coronas of certain compactifications (large scale compactifications in case of coarse spaces). A future paper will demonstrate the equivalence of Specker's definition of ends of groups and our definition of them.

W. Dicks and M. J. Dunwoody \cite{DD} also consider ends of non-finitely generated groups. In particular, they prove the following result that is a generalization of the famous theorem of Stallings \cite{St}:

\begin{Theorem}
A group $G$ has infinitely many ends if and only if 
one of the following conditions holds:\\
(i) $G$ is countably infinite and locally finite,\\
(i) $G$ can be expressed as an amalgamated free product $A\ast_CB$ or an HNN extension $A\ast_C$, where $C$ is a finite subgroup of $A$ and $B$ such that $[A : C]\ge 3$ and $[B : C]\ge 2$.
\end{Theorem}

The authors are grateful to Ross Geoghegan and Mike Mihalik for their help in understanding classical theory of ends of finitely generated groups.

\section{Ends of coarse spaces}

In this section we generalize the concept of Freudenthal ends to arbitrary coarse spaces. See \cite{JDEnds} for other ways to introduce ends in coarse spaces.

We follow a description of coarse spaces (quite often our terminology is that of \textbf{large scale spaces}) as in \cite{DH}. It is equivalent to Roe's definition of those spaces in \cite{Roe lectures}.

Recall that a \textbf{star} $st(x,U)$ of $x\in X$ with respect to a family $\mathcal{U}$ of subsets of $X$ is defined as the union of $U\in \mathcal{U}$ containing $x$. If $A\subset X$, then $st(A,\mathcal{U}):=\bigcup\limits_{x\in A}st(x,\mathcal{U})$. Given two families $\mathcal{U},\mathcal{V}$ of subsets of $X$,
$st(\mathcal{U},\mathcal{V})$ is defined as the family $st(A,\mathcal{V})$, $A\in \mathcal{U}$.

\begin{Definition}
A \textbf{large scale space} is a set $X$ equipped with a family $\mathbb{LSS}$ of covers (called \textbf{uniformly bounded} covers) satisfying the following two conditions:\\
1. 
$st(\mathcal{U},\mathcal{V})\in \mathbb{LSS}$ if $\mathcal{U},\mathcal{V}\in \mathbb{LSS}$.\\
2. If $\mathcal{U}\in \mathbb{LSS}$ and every element of $\mathcal{V}$ is contained in some element of $\mathcal{U}$, then $\mathcal{V}\in \mathbb{LSS}$.

Sets which are contained in an element of $\mathcal{U}\in \mathbb{LSS}$ are called \textbf{bounded}.

In this paper we consider only large scale spaces that have the \textbf{lowest form of coarse connectivity}. Namely, the union of two bounded subsets of $X$ is always bounded.
\end{Definition}

\begin{Definition}
The subsets $A$ and $C$ of a large scale space $X$ are \textbf{coarsely disjoint} if for every uniformly bounded cover $\mathcal{U}$ of $X$ the set $st(A,\mathcal{U})\cap st(C,\mathcal{U})$ is bounded. $A$ is \textbf{coarsely clopen} if $A$ and $A^c$ are coarsely disjoint.

A \textbf{non-trivial coarsely clopen} subset $A$ of a large scale space $X$ (an NCC-set for short) is one that is not bounded and $A^c$ is not bounded.
\end{Definition}

\begin{Lemma}\label{StLemma}
$st(A_1\cap A_2,\mathcal{U})\cap st((A_1\cap A_2)^c,\mathcal{U})\subset
st(A_1,\mathcal{U})\cap st((A_1)^c,\mathcal{U})\cup st(A_2,\mathcal{U})\cap st((A_2)^c,\mathcal{U})$.
\end{Lemma}
\begin{proof}
Suppose $x\in st(A_1\cap A_2,\mathcal{U})\cap st((A_1\cap A_2)^c,\mathcal{U})$. There is $y\in A_1\cap A_2$ satisfying $x\in st(y,\mathcal{U})$
and there is $z\in A_1^c\cup A_2^c$ satisfying $x\in st(z,\mathcal{U})$.
Thus either $x\in st((A_1)^c,\mathcal{U})$ or $x\in st((A_2)^c,\mathcal{U})$
and we are done.
\end{proof}

\begin{Corollary}
\label{IntersectionOFCCs}
The intersection of two coarsely clopen subsets of $X$ is coarsely clopen.
\end{Corollary}
\begin{proof}
Apply \ref{StLemma}. Notice that it is exactly here we need the union of two bounded subsets of $X$ be always bounded.
\end{proof}

\begin{Definition}
A topology on $X$ is \textbf{compatible} with the large scale structure on $X$ if there is a uniformly bounded cover of $X$ consisting of open subsets of $X$. 
\end{Definition}

\begin{Observation}
The simplest non-trivial topology compatible with a large scale structure is the discrete topology.
\end{Observation}

\begin{Definition}
A \textbf{topological large scale space} is a set equipped with large scale structure and with a compatible topology. Additionally, we assume that the coarse structure is \textbf{coarsely connected}, i.e. the union of two bounded subsets of $X$ is bounded.
\end{Definition}

\begin{Example}
Every metric space $(X,d)$ has a natural topological large scale structure  $ \mathbb{LSS}_d=\{\mathcal{U}_r | r>0 \}$, where $\mathcal{U}_r$ is the family of all subsets of $X$ whose diameter is at most $r$. Notice that for any subset $A$ of $X$ and any uniformly bounded cover $\mathcal{U}_r$, one has  $st(A,\mathcal{U}_r)=B(A,r)$; in particular, a subset $A$ of $X$ is coarsely clopen if one of the following conditions holds:\\
1. For all $r>0$, the subset $B(A,r)\cap A^c$ is bounded,\\
2. For all $r>0$, the subset $A\cap B(A^c,r)$ is bounded,\\
3. For all $r>0$, there is a bounded subset $K_{r}$ of $X$ such that $B(A\setminus K_{r},r)\cap B(A^c,r)$ is empty,\\
4. For all $r>0$, there is a bounded subset $K_{r}$ of $X$ such that $B(A,r)\cap B(A^c\setminus K_{r},r)$ is empty.
\end{Example}

\begin{Lemma}\label{ExtendingNCCLemma}
If $A$ is a (non-trivial) coarsely clopen subset of $X$, then $st(A,\mathcal{U})$ is a (non-trivial) coarsely clopen subset of $X$ for each uniformly bounded cover $\mathcal{U}$ of $X$. 
\end{Lemma}
\begin{proof}
Observe that if $A$ is a non-trivial coarsely clopen subset of $X$, then $C:=(st(A,\mathcal{U}))^c$ cannot be bounded. Indeed, in that case $A^c\subset st(st(C,\mathcal{U}),\mathcal{U})$ would be bounded.

Notice $st(st(A,\mathcal{U}),\mathcal{V})\subset st(A,st(\mathcal{U},\mathcal{V}))$ for any two covers $\mathcal{U},\mathcal{V}$.
Therefore $st(st(A,\mathcal{U}),\mathcal{V})\cap st(st(A^c,\mathcal{U}),\mathcal{V})\subset st(A,st(\mathcal{U},\mathcal{V}))\cap st(A^c,st(\mathcal{U},\mathcal{V}))$. Since
$st(A,\mathcal{U})^c\subset A^c\subset st(A^c,\mathcal{U})$ the proof is completed.
\end{proof}

\begin{Lemma}\label{ShrinkingNCCLemma}
If $A$ is a coarsely clopen subset of $X$, then a subset $C$ of $A$ is coarsely clopen provided
 $A\subset st(C,\mathcal{V})$ for some uniformly bounded cover $\mathcal{V}$ of $X$.
\end{Lemma}
\begin{proof}
Observe $C':=(st(st(A^c,\mathcal{V}),\mathcal{V}))^c\subset C$ is coarsely clopen by \ref{ExtendingNCCLemma} and $B:=C\setminus C'\subset A\cap st(st(A^c,\mathcal{V}),\mathcal{V})
\subset st(st(A,\mathcal{V})\cap st(A^c,\mathcal{V}),\mathcal{V})$ is bounded as $st(A,\mathcal{V})\cap st(A^c,\mathcal{V})$
is bounded. Adding a bounded set $B$ to a coarsely clopen subset
preserves being coarsely clopen as can be easily seen.
\end{proof}

\begin{Definition}
An \textbf{end} of a large scale space $X$ is a family $E$ of unbounded and coarsely clopen subsets of $X$ that is maximal with respect to the property of all finite intersections being unbounded. The set of all ends of $X$ is denoted by $Ends(X)$.
\end{Definition}

\begin{Proposition}\label{OpenEndsProp}
If $X$ is a topological large scale space and $A$ belongs to an end $E$ of $X$, then there is an open $V\in E$ contained in $A$.
Consequently,
two ends $E$ and $E'$ are equal if and only if $\{U\in E | U \mbox{ is open}\}=\{V\in E' | V \mbox{ is open}\}$.
\end{Proposition}
\begin{proof}
Let $\mathcal{U}$ be a uniformly bounded cover of $X$ consisting of open subsets. 
Notice $cl(B)\subset st(B,\mathcal{U})$ is bounded for every bounded subset of $X$. 

Consider $B:= st(A^c,\mathcal{U})\cap st(A,\mathcal{U})$
and observe $V:=A\setminus cl(B)$ is open. Indeed, if $x\in V$,
then $st(x,\mathcal{U})\setminus cl(B)\subset V$ as otherwise
there is $y\in (st(x,\mathcal{U})\setminus cl(B))\setminus A$ resulting in $x\in st(A^c,\mathcal{U})\cap A\subset B$, a contradiction.

Suppose $\{U\in E | U \mbox{ is open}\}=\{V\in E' | V \mbox{ is open}\}$
and $C\in E'\setminus E$. There is $A\in E$ such that $C\cap A$ is bounded. 
Choose $V\subset C$, $V\in E'$ and open. 
Now, $V\in E$
contradicting $A\cap V$ being bounded.
\end{proof}

\begin{Notation}
Let $X$ be a large scale space. If $U$ is a coarsely clopen subset of $X$, we define $U_{end}:=\{E\in Ends(X)|U\in E\}$ and $\widetilde{U}_{end}:=U_{end}\cup U$.

\end{Notation}
\begin{Proposition}
Let $\mathcal{T}$ be the topology of a topological large scale space $X$. The collection $\widetilde{\mathcal{B}}=\mathcal{T}\cup \{\widetilde{U}_{end}|$$U$ open coarsely clopen subset of $X$$\}$ is a basis for a topology $\mathcal{T}_{end}$ on $X\cup Ends(X)$ that extends the topology $\mathcal{T}$. 
\end{Proposition}
\begin{proof}
Clearly $\widetilde{\mathcal{B}}$ covers $X\cup Ends(X)$ as either $X$ is unbounded open coarsely clopen and $\widetilde{X}_{end}=X\cup Ends(X)$ or $X$ is bounded and $Ends(X)=\emptyset$. Now, let $\widetilde{U}_{end},\widetilde{V}_{end}\in \widetilde{\mathcal{B}}$. Notice $\widetilde{U}_{end}\cap \widetilde{V}_{end}=(\widetilde{U\cap V})_{end}\in \widetilde{\mathcal{B}}$. Finally, every bounded element of $\mathcal{T}$ is contained in $\widetilde{\mathcal{B}}$ which implies that $\mathcal{T}\subset \mathcal{T}_{end}$.
\end{proof}

\begin{Corollary}
Let $\mathcal{T}$ be the topology of a topological large scale space $X$. The collection $\mathcal{B}_{\mathcal{T}}=\{{U}_{end}|$$U$ open coarsely clopen subset of $X$$\}$ is a basis for the topology of $Ends(X)$ as a subspace of $X\cup Ends(X)$.
\end{Corollary}

\begin{Corollary}
Given two compatible topologies on a large scale space $X$, the induced topologies on $Ends(X)$ coincide.
\end{Corollary}
\begin{proof}
Use \ref{OpenEndsProp}.
\end{proof}

\begin{Definition}
Let $X$ be a large scale space. Two subsets $A,C\subseteq X$ are said to be \textbf{coarsely identical} if their symmetric difference $A\Delta C:=(A\setminus C)\cup (C\setminus A)$ is a bounded subset of $X$.
\end{Definition}

\begin{Example} 
Let $A$ be a coarsely clopen subset of a large scale space $X$, and $U$ be a uniformly bounded cover of $X$. Then $A$ and $st(A,\mathcal{U})$ are coarsely identical.
\end{Example}

\begin{Lemma}\label{CoarselyIdentical}
Let $E\in Ends(X)$ and $U\in E$, then $int(U), cl(U)\in E$.
\end{Lemma}
\begin{proof}
Use \ref{OpenEndsProp} to see that $int(U)\in E$.
It suffices to show that $int(U)$ and $cl(U)$ are coarsely identical. Let $\mathcal{U}$ be a uniformly bounded open cover of $X$. Notice that $cl(U)\setminus int(U)\subseteq st(U,\mathcal{U})\cap st(U^c,\mathcal{U})$, and that $st(U,\mathcal{U})\cap st(U^c,\mathcal{U})$ is bounded as $U$ is coarsely clopen.
\end{proof}

\begin{Proposition}
Let $X$ be a topological large scale space. If $X$ is Hausdorff, then the topological space $X\cup Ends(X)$ is Hausdorff. 
\end{Proposition}
\begin{proof}
We consider two cases:\\
\textbf{Case 1:} Let $x\in X$ and $E\in Ends(X)$ and let $\mathcal{U}$ be a uniformly bounded open cover of $X$. Choose $U\in \mathcal{U}$ such that $x\in U$. Notice that $cl(U)$ is a bounded subset and hence $V:=X\setminus cl(U)$ is open and coarsely clopen contained in $E$. Hence, $U$ and $\widetilde{V}_{end}$ are disjoint open neighborhood of $x$ and $E$, respectively. \\
\textbf{Case 2:} If $E_{1}, E_{2}\in Ends(X)$ are two distinct ends of $X$, then we can find $A\in E_{1}$ and $B\in E_{2}$ such that $A\cap B$ is a bounded subset of $X$. Without loss of generality, we may assume that $A\cap B=\emptyset$. Therefore, $U=int(A)\in E_{1}$ and  $V=int(B)\in E_{2}$ are disjoint and hence $\widetilde{U}_{end}$ and $\widetilde{V}_{end}$ are disjoint neighborhoods of $E_{1}$ and $E_{2}$, respectively. 
\end{proof}

$X$ is an open dense subspace of $X\cup Ends(X)$. Moreover, $(X\cup Ends(X),\mathcal{T}_{end})$ is Hausdorff whence $(X,\mathcal{T})$ is Hausdorff. The question then arises: is $(X\cup Ends(X),\mathcal{T}_{end})$ a compactification of $(X,\mathcal{T})$? The answer is positive in a coarse sense!

\begin{Definition}
Let $Y$ be a topological space and $X\subseteq Y$ is a topological subspace equipped with a large scale structure compatible with its topology. $Y$ is \textbf{large scale compact} if for any open cover $\{U_s\}_{s\in S}$ of $Y$, there is a finite subset  $F$ of $S$ such that $Y\setminus \bigcup\limits_{s\in F}U_s$ is a bounded subset of $X$. $Y$ is a \textbf{large scale compactification} of $X$ if in addition to being large scale compact, $Y$ is Hausdorff and $X$ is an open dense subspace of $Y$(see \cite{JDUnifying}).
\end{Definition}

\begin{Lemma}\label{LargeScleCompact}
For any family $\{U^{s}\}_{s\in I}$ of coarsely clopen subsets of $X$ such that $Ends(X)\subset \bigcup\limits_{s\in I}U^{s}_{end}$, there is a finite subset  $F$ of $I$ such that $X\setminus \bigcup\limits_{s\in F}U^{s}$ is a bounded subset of $X$.
\end{Lemma}

\begin{proof}
Let $\mathcal{F}$ be the collection of all finite subsets of $I$. Seeking contradiction assume that for any $F\in \mathcal{F}$, $A_{F}=X\setminus \bigcup\limits_{s\in F}U^{s}$ is unbounded. The collection $\{A_{F}|F\in \mathcal{F}\}$ is contained in some end $E\in Ends(X)$. Hence, $\{A_{F}|F\in \mathcal{F}\}\subset E\in U^{s}_{end}$ for some $s\in I$ which implies that $U^{s}\in E$ and $X\setminus U^{s}\in E$, a contradiction. 
\end{proof}

\begin{Theorem}
Let $X$ be a Hausdorff topological large scale space, the topological space $X\cup Ends(X)$ is a large scale compactification of $X$. Furthermore, $Ends(X)$ is a compact totally disconnected subspace of $X\cup Ends(X)$.
\end{Theorem}
\begin{proof}
Let $\{O^{s}\}_{s\in I}$ be an open cover $X\cup Ends(X)$, we can find a subset $J\subset I$ such that $\{\widetilde{U^{s}}_{end}\}_{s\in J}$ of $\{O^{s}\}_{s\in I}$ that covers $Ends(X)$. In particular, $Ends(X)\subset \bigcup\limits_{s\in J}U^{s}_{end}$. By the above lemma, there is a finite subset $F$ of $J$ such that $X\setminus \bigcup\limits_{s\in F}U^{s}$ is a bounded subset of $X$. We claim that $Ends(X)\subset \bigcup\limits_{s\in F}U^{s}_{end}$. To this end, notice that if $U\subset X$ is an arbitrary coarsely clopen subset of $X$ and $V=X\setminus U$, then $Ends(X)=U_{end}\cup V_{end}$. Indeed, if $E\in Ends(X)\setminus (U_{end}\cup V_{end})$, then we can find $A,B\in E$ such that both $A\cap U$ and $B\cap V$ are bounded. Hence, there exist $K,L$ bounded subsets of $X$ such that $A\subset V\cup K$ and $B\subset U\cup L$. In particular, $A\cap B$ is bounded, a contradiction. Now, seeking contradiction assume that $E\in Ends(X)\setminus \bigcup\limits_{s\in F}U^{s}_{end}$, then $E\notin U^{s}_{end}$ for all $s\in F$. Therefore, by above observation, $X\setminus U^{s}\in E$ for all $s\in F$ which implies that $\bigcap\limits_{s\in F}(X\setminus U^{s})=X\setminus \bigcup\limits_{s\in F}U^{s}\in E$, a contradiction. This shows that $Ends(X)$ is compact and $X\cup Ends(X)$ is large scale compact as $(X\cup Ends(X))\setminus \bigcup\limits_{s\in F}\widetilde{U^{s}}_{end}$ is bounded.\\ Finally, we show that $Ends(X)$ is totally disconnected. Let $E_{1}, E_{2}\in Ends(X)$ be two distinct ends of $X$, then we can find $A\in E_{1}$ and $B\in E_{2}$ such that $cl(A)\cap B$ is a bounded subset of $X$. Without loss of generality, we may assume that $cl(A)\cap B=\emptyset$. Clearly, $U=int(A)\in E_{1}$ and $V=X\setminus cl(A)\in E_{2}$. Moreover, $Ends(X)=U_{end}\cup V_{end}$ and $U_{end}\cap V_{end}=\emptyset$.
\end{proof}

Let us show that $Ends(X)$ is a coarse invariant of a space. 

\begin{Definition}
Let $\varphi,{\varphi}':X\to Y$ be maps between large scale spaces.\\
$\bullet$ $\varphi$ and ${\varphi}'$ are \textbf{close} if there is a uniformly bounded cover $\mathcal{U}$ of $Y$ such that
$\varphi(x)\in st({\varphi}'(x),\mathcal{U})$ for each $x\in X$.\\
$\bullet$ $\varphi $ is \textbf{coarse} if $\varphi^{-1}(K)$ is bounded for each bounded subset $K$ of $Y$.\\
$\bullet$ $\varphi$ is \textbf{large scale continuous} if $\varphi(\mathcal{U})$ is a uniformly bounded cover of $\varphi(X)$ for each uniformly bounded cover $\mathcal{U}$ of $X$.\\
$\bullet$ $\varphi$ is a \textbf{coarse equivalence} if there is a coarse, large scale continuous map $\psi:Y\to X$ such that $\psi\circ \varphi$ is close to $id_X$ and $\varphi\circ \psi$ is close to $id_Y$. In a such case, $\varphi$ and $\psi$ are \textbf{coarse inverses} of each others; $X$ and $Y$ are \textbf{coarsely equivalent}. 
\end{Definition}

\begin{Lemma}\label{ImageOfAnEnd}
Suppose $f:X\to Y$ is a coarse large scale continuous function of topological large scale spaces and $E\in Ends(X)$. Then, using the closure operation $cl$ in $Y\cup Ends(Y)$, the following hold:\\
a.
If $G\in \bigcap\limits_{A\in E} cl(f(A))$ is an end of $Y$ and $V\in G$ is an open and coarsely clopen subset of $Y$, then $f^{-1}(V)\in E$.\\
b. $\bigcap\limits_{A\in E} cl(f(A))$
is a singleton belonging to $Ends(Y)$.
\end{Lemma}
\begin{proof}
Observe $f^{-1}(D)$ is coarsely clopen in $X$ if $D$ is coarsely clopen in $Y$.
Indeed, if $B:=st(f^{-1}(D),\mathcal{U})\cap st(f^{-1}(D)^c,\mathcal{U})$ is bounded for some
uniformly bounded cover $\mathcal{U}$ of $X$, then 
$st(D,\mathcal{U}_f)\cap st(D^c,\mathcal{U}_f)\subset f(B)$ is bounded, a contradiction. Here, $ \mathcal{U}_f$ is $f(\mathcal{U})$ union singletons outside of $f(X)$.

$a$. If $f^{-1}(V)\notin E$, then there is
$A\in E$ such that $A\cap f^{-1}(V)$ is bounded, hence its image $B:=f(A)\cap V$ is bounded. Put $A_1=A\setminus f^{-1}(B)$ to get $f(A_1)\subset Y\setminus V$ contradicting $E\in cl(f(A_1))$. Thus $f^{-1}(V)\in E$.

$b$.
$\bigcap\limits_{A\in E} cl(f(A))$ cannot be empty and it is contained in $Ends(Y)$. First of all, if $y\in Y$, then there is an open bounded neighborhood
$V$ of $y$ in $Y$, so $f^{-1}(V)$ is bounded in $X$. Hence $A:=X\setminus 
f^{-1}(V)\in E$, $f(A)\subset Y\setminus V$, and $cl(f(A))$ misses $V$.
Secondly, assume that $\bigcap\limits_{A\in E} cl(f(A))=\emptyset$. By \ref{LargeScleCompact}, there exist finitely many $A_i\in E$, $i\leq n$, such that $\bigcap\limits_{i=1}^n cl(f(A_i))$ is bounded. However, $A:=\bigcap\limits_{i=1}^n A_i\in E$,
and $f(A)$ is bounded, a contradiction.
Finally, if $E_1,E_2\in \bigcap\limits_{A\in E} cl(f(A))$ and $E_1\ne E_2$,
then there exists two disjoint open and coarsely clopen subsets $V_1\in E_1$
and $V_2\in E_2$. By a) $f^{-1}(V_1)\in E$ and $f^{-1}(V_2)\in E$, a contradiction.
\end{proof}

\begin{Corollary}
If $f:X\to Y$ is a coarse large scale continuous function of topological large scale spaces and $E\in Ends(X)$, then it induces a continuous function $f_{end}:Ends(X)\to Ends(Y)$ defined by $f_{end}(E)\in \bigcap\limits_{A\in E} cl(f(A))$. \\
a. If $f,g:X\to Y$ are close, then $f_{end}=g_{end}$.\\
b. If $f$ is continuous, then $f\cup f_{end}:X\cup Ends(X)\to Y\cup Ends(Y)$
is continuous.
\end{Corollary}
\begin{proof}
Given an open, coarsely clopen subset $V$ of $Y$, notice that
$f_{end}^{-1}(V_{end})\subset U_{end}$, where $U=f^{-1}(V)$. 
Indeed, if $f_{end}(E)\in V_{end}$ then $E\in U_{end}$ by a) of \ref{ImageOfAnEnd}.
That proves continuity of $f_{end}$ and of $f\cup f_{end}$ if $f$ is continuous.

Suppose there is an open uniformly bounded cover $\mathcal{U}$ of $Y$ such that
$f(x)\in st(g(x),\mathcal{U})$ for each $x\in X$.
If there is $E\in Ends(X)$ such that $f_{end}(E)\ne g_{end}(E)$, then
we can choose open subsets $U\in f_{end}(E), V\in g_{end}(E)$ so that
$st(U,\mathcal{U})\cap V=\emptyset$.
Notice $f^{-1}(U)\in E$ and $g^{-1}(V)\in E$ by a) of \ref{ImageOfAnEnd}.
Now, there is $x\in f^{-1}(U)\cap g^{-1}(V)$, so $f(x)\in U$, $g(x)\in V$
contradicting $st(U,\mathcal{U})\cap V=\emptyset$.
\end{proof}

\begin{Corollary}\label{CoarselyEquivalentSpacesHaveHomeomorphicEndsSpaces}
If two topological large scale spaces $X$ and $Y$ are coarsely equivalent, then 
$Ends(X)$ is homeomorphic to $Ends(Y)$.
\end{Corollary}

\begin{proof}
It suffices to show that if $f:X\to Y$ and $g:Y\to Z$ are coarse large scale continuous functions of topological large scale spaces, then $(g\circ f)_{end}=f_{end}\circ g_{end}$. Indeed, let $E\in Ends(X)$ such that $(g\circ f)_{end}(E)\neq f_{end}\circ g_{end}(E)$. There exist open disjoint subsets $U\in (g\circ f)_{end}(E)$ and $V\in f_{end}\circ g_{end}(E)$. Now, $(g\circ f)^{-1}(U)\in E$ and $f^{-1}(V)\in g_{end}(E)$ by a) of \ref{ImageOfAnEnd}. In particular, $(g\circ f)^{-1}(U), (g\circ f)^{-1}(V)\in E$, a contradiction. 
\end{proof}

\section{Comparison of coarse ends and Freundenthal ends}

In this section we are concentrating on a relationship between Freudenthal ends and coarse ends. More specifically, we are interested in cases where there is a one-to-one correspondence between those ends in the following sense: each Freudenthal end is contained in a unique coarse end and each coarse end contains a Freudenthal end. To accomplish it, we need a large scale analog of local connectedness. Also, we need to generalize Freundenthal ends to non-locally compact spaces. See \cite{DickmanMcCoy} for a theory of Freundenthal compactifications for general topological spaces.

\begin{Definition}
Suppose $X$ is a locally connected topological space that is a union of an increasing sequence of closed subspaces $\{K_i\}_{i\ge 1}$
such that $K_i\subset int(K_{i+1})$ for each $i\ge 1$. Let $\mathcal{K}$ be the bornology generated by $\{K_i\}_{i\ge 1}$, i.e. all subsets $B$ of $X$ such that $B\subset K_i$ for some $i\ge 1$.
A \textbf{Freundenthal end} of $X$ with respect to $\mathcal{K}$ is a decreasing sequence $\{C_i\}_{i\ge 1}$, where each $C_i$ is a non-empty component of $X\setminus K_i$.

Notice each $C_i$ is unbounded as otherwise it is contained in some $K_k$, $k > i$, and $C_k\subset C_i\cap (X\setminus K_k)=\emptyset$, a contradiction. 
\end{Definition}

\begin{Definition}
A large scale space $X$ is \textbf{large scale chain-connected} if there is a uniformly bounded cover $\mathcal{U}$ of $X$ such that every uniformly bounded cover $\mathcal{V}$ of $X$ is a refinement of a uniformly bounded cover $\mathcal{W}$ that consists of $\mathcal{U}$-connected sets (that means any two points $x,y\in W\in \mathcal{W}$
can be connected by a chain of points $x_1=x,\ldots, x_n=y$ with the property that for any $i < n$ there is $U\in \mathcal{U}$ containing both $x_i$ and $x_{i+1}$).
\end{Definition}

\begin{Definition}
Suppose $X$ is a large scale space that is large scale chain-connected via a uniformly bounded cover $\mathcal{U}$ and the bornology $\mathcal{B}$ of $X$ has an increasing sequence $\{B_i\}_{i\ge 1}$ of bounded subsets of $X$ that serves as a basis of it. A \textbf{Freundenthal end} of $X$ with respect to $\mathcal{B}$ is a decreasing sequence $\{C_i\}_{i\ge 1}$, where each $C_i$ is a non-empty $\mathcal{U}$-component of $X\setminus B_i$. 
\end{Definition}

\begin{Theorem}\label{MainComparisonTheorem}
Suppose $X$ is a large scale space that is large scale chain-connected via a uniformly bounded cover $\mathcal{U}$ and has the property that for any bounded subset $B$ of $X$ the union of all bounded $\mathcal{U}$-components of $X\setminus B$ is bounded and there are only finitely many unbounded $\mathcal{U}$-components of $X\setminus B$. 
If there is an increasing sequence $\{B_i\}_{i\ge 1}$ of bounded subsets of $X$ that is a basis for all bounded subsets of $X$, then the coarse ends of $X$ are in one-to-one correspondence with Freundenthal ends of $X$ with respect to its bornology.
\end{Theorem}
\begin{proof}
\textbf{Claim 1:} If $C$ is a union of $\mathcal{U}$-components of $X\setminus B$ for some bounded subset $B$ of $X$, then $C$ is coarsely clopen.\\
\textbf{Proof of Claim 1}: Given a uniformly bounded cover $\mathcal{V}$ of $X$ consisting of $\mathcal{U}$-connected subsets of $X$, the set $st(C,\mathcal{V})\cap st(C^c,\mathcal{V})$ is contained in $st(B,\mathcal{V})$.
Indeed, if $x\notin B$ belongs to $st(C,\mathcal{V})\cap st(C^c,\mathcal{V})$,
then there is $V\in \mathcal{V}$ containing $x$ and intersecting both $C$ and $C^c$.
Therefore $V\setminus B$ must be contained in either $C$ or $C^c$,
so $V\cap B\ne\emptyset$ and $x\in st(B,\mathcal{V})$.

\textbf{Claim 2:} If $C$ is coarsely clopen, then every union of $\mathcal{U}$-components of $C$ is coarsely clopen.\\
\textbf{Proof of Claim 2}: $B:=st(C,\mathcal{U})\cap st(C^c,\mathcal{U})$ is bounded
and $C\setminus B$ is a union of $\mathcal{U}$-components of $X\setminus B$.
By Claim 1, $C\setminus B$ is coarsely clopen, so so is $C$.

\textbf{Proof of \ref{MainComparisonTheorem}}:
Consider a Freundenthal end $\{C_i\}_{i\ge 1}$ of $X$ with respect to its bornology. $\{C_i\}_{i\ge 1}$ is contained in some coarse end of $X$ by Claim 1. $\{C_i\}_{i\ge 1}$ cannot be contained in two different coarse ends $E_1$, $E_2$. Indeed, choose $A_1\in E_1$, $A_2\in E_2$ whose stars with respect to $\mathcal{U}$ are disjoint. 
There is $k > 1$ such that $st(A_1,\mathcal{U})\cap st(A_1^c,\mathcal{U})\subset B_k$. Notice that the union $D_1$ of $\mathcal{U}$-components of $X\setminus B_k$ containing points from $C_k\cap A_1$ does not intersect 
$C_k\cap A_1^c$. That means $C_k$ can be expressed as the union of $D_1$ and $C_k\setminus D_1$, both unbounded unions of $\mathcal{U}$-components of $X\setminus B_k$. There is $j > k$ such that either $C_j\subset D_1$ or
$C_j\subset C_k\setminus D_1$.
In the first case $C_j\notin E_2$ and in the second case $C_j\notin E_1$, a contradiction.

Given a coarse end $E$ of $X$ and given $A\in E$ there is $i\ge 1$ such that
$st(A,\mathcal{U})\cap st(A^c,\mathcal{U})\subset B_i$. Consider $j > i$ satisfying
$st(B_i,\mathcal{U})\subset B_j$. Notice $A\setminus B_j$ is a union of $\mathcal{U}$-components of $X\setminus B_j$. Therefore there is an umbounded $U$-component of $X\setminus B_j$ belonging to $E$. That means, starting from some $k > i$, $E$ contains exactly one unbounded component $C_n$ of $X\setminus B_n$ for all $n\ge k$. Those components can be easily extended to a Freundenthal end of $X$ contained in $E$.
\end{proof}

\begin{Theorem}
Suppose $X$ is a connected, locally compact space that is locally connected and $X$ is the union of an increasing sequence of compact subspaces $\{K_i\}_{i\ge 1}$
such that $K_i\subset int(K_{i+1})$ for each $i\ge 1$.
There is a large scale structure on $X$ such that coarse ends of $X$ are in are in one-to-one correspondence with Freundenthal ends of $X$.
\end{Theorem}
\begin{proof}
Consider all open covers $\mathcal{V}$ of $X$ consisting of open connected subsets of $X$ with the property that $st(K,\mathcal{V})$ is pre-compact for each compact subset $K$ of $X$. The large scale structure on $X$ consists of all refinements of such covers.

The basic cover $\mathcal{U}$ consists of components of sets $int(K_{i+2})\setminus K_i$ for $i\ge 0$, where we put $K_0=\emptyset$. That implies the cover of $X$ consisting of singletons is indeed uniformly bounded and the large scale is legit.

Given two covers $\mathcal{V}, \mathcal{W}$ of $X$ consisting of open connected subsets of $X$ with the property that $st(K,\mathcal{V})$ is pre-compact for each compact subset $K$ of $X$, elements of the cover $st(\mathcal{V},\mathcal{W})$ are $\mathcal{W}$-connected, so the large scale is large scale chain-connected.

Given a bounded subset $B$ of $X$, cover $cl(B)$ by finitely many elements $U_1,\ldots, U_k$ of $\mathcal{U}$. Any $\mathcal{U}$-chain joining $x\in X\setminus B$ to $b\in B$ must intersect one of $U_i$. Therefore $X\setminus B$ has only finitely many $st(\mathcal{U},\mathcal{U})$-components.
Apply \ref{MainComparisonTheorem}.

\end{proof}

\section{Large scale groups}

In this section we introduce the concept of a large scale group that generalizes the following classes of groups:\\
1. Finitely generated groups with word metrics,\\
2. Countable groups with proper left-invariant metrics,\\
3. Locally compact topological groups.

Ideally, the name of those groups should be coarse groups but it has been already used in literature for similar but different objects (see \cite{LV}).

See \cite{BDM} for a discussion of large scale structures on a group induced by right-invariant metrics versus left-invariant metrics.

Recall that a \textbf{bornology} on a set $X$ is a cover of $X$ that is stable under inclusion and is stable under finite unions.

\begin{Definition}
A \textbf{large scale group} is a group $G$ equipped with a large scale structure $\mathcal{LSS}$ induced by a bornology $\mathcal{B}$. That means every uniformly bounded cover in $\mathcal{LSS}$ is a refinement of $\{g\cdot B\}_{g\in G}$ for some $B\in \mathcal{B}$.
\end{Definition}

\begin{Lemma}
A bornology $\mathcal{B}$ on a group $G$ induces a large scale structure on $G$ if and only if $\mathcal{B}$ is stable under inverses and products.
\end{Lemma}
\begin{proof}
Notice $\mathcal{B}$ is stable under inverses and products if and only if for all non-empty $B_1,B_2\in \mathcal{B}$ the set $(B_1\cdot B_2^{-1})\cdot B_2$ belongs to $\mathcal{B}$.

Let $\mathcal{U}:=\{g\cdot B_2\}_{g\in G}$. Notice $(B_1\cdot B_2^{-1})\cdot B_2=\bigcup\limits_{g\in G} B_1\cap (g\cdot B_2)=st(B_1,\mathcal{U})$.
Conversely, if $B_3:=\bigcup\limits_{g\in G} B_1\cap (g\cdot B_2)\in \mathcal{B}$,
then $st(h\cdot B_1,\mathcal{U})\subset h\cdot B_3$ for all $h\in G$.
\end{proof}

\begin{Corollary}
Large scale groups include the following classes of groups:\\
1. Finitely generated groups with word metrics,\\
2. Countable groups with proper left-invariant metrics,\\
3. Locally compact topological groups.
\end{Corollary}
\begin{proof}
In cases 1) and 2) the bornologies consist of all finite subsets of $G$.

In case 3) the bornology $\mathcal{B}$ consist of all subsets of compact sets in $G$. Indeed, given $B_1,B_2\in \mathcal{B}$ the set
 $(B_1\cdot B_2^{-1})\cdot B_2$ is pre-compact (its closure is compact).
\end{proof}

\begin{Lemma}\label{CoversOfSubgroupsLemma}
Given a subgroup $H$ of a group $G$ the restriction of a cover
$\{g\cdot B\}_{g\in G}$ to $H$ is a refinement of the cover $\{h\cdot ((B^{-1}\cdot B)\cap H)\}_{h\in H}$ of $H$.
\end{Lemma}
\begin{proof}
Since $B$ is non-empty, so is $ B^{-1}\cdot B$. For each $g\in G$ such that $H\cap (g\cdot B)\ne\emptyset$ pick $h_g\in H\cap (g\cdot B)$.
If $h\in (g\cdot B)\cap H$,
then $b:=g^{-1}\cdot h\in B$ and $b_g:=g^{-1}\cdot h_g\in B$, so $h=g\cdot b=h_g\cdot b_g^{-1}\cdot b\in h_g\cdot ((B^{-1}\cdot B)\cap H)$.
\end{proof}

\begin{Corollary}
Given a subgroup $H$ of a large scale group $G$ the induced large scale on $H$ equals the large scale generated by the restriction of the bornology of $G$ to $H$.
\end{Corollary}

\begin{Theorem}
A large scale group $G$ is of asymptotic dimension $0$ if and only if for every bounded subset $B$ of $G$ the subgroup $<B>$ of $G$ generated by $B$ is bounded.
\end{Theorem}
\begin{proof}
See \cite{BDHM} for a discussion of asymptotic dimension $0$.
Given a bounded subset $B$ of $G$ consider a set of elements $\{g_i\}_{i\in J}$ representing all cosets $g\cdot <B>$. Notice $\{g_i\cdot <B>\}_{i\in J}$
is a uniformly bounded of $G$ consisting of mutually disjoint sets for which
$\{g\cdot B\}_{g\in G}$ is a refinement.

Suppose $G$ is of asymptotic dimension $0$ and $B\subset G$ is bounded and symmetric.
Choose a uniformly bounded cover $\{U_i\}_{i\in J}$ consisting of mutually disjoint sets for which
$\{g\cdot B\}_{g\in G}$ is a refinement. Let $1_G\in U_0$. If $h\in U_0$,
then $h\cdot B\subset U_0$ as otherwise $h\cdot B\subset U_i$ for some $i\ne 0$ and $h\in U_0\cap U_i$, a contradiction. Consequently, any finite product of elements of $B$ belongs to $U_0$. Hence $<B>\subset U_0$ is bounded.
\end{proof}

\begin{Definition}
A subgroup $H$ of a large scale group $G$ is of \textbf{bounded index} in $G$ if there is a bounded subset $B$ of $G$ such that $B\cdot H=G$.
\end{Definition}

\begin{Proposition}\label{BoundedIndexSubgroups}
A subgroup $H$ of a large scale group $G$ is of bounded index in $G$ if and only if the inclusion $H\to G$ is a coarse equivalence.
\end{Proposition}
\begin{proof}
If $B\cdot H=G$, then $H\cdot B^{-1}=G$, so the star of $H$ with respect to 
$\{g\cdot B^{-1}\}_{g\in G}$ equals $G$ and the inclusion $H\to G$ is a coarse equivalence.

Conversely, if the star of $H$ with respect to 
$\{g\cdot B\}_{g\in G}$ equals $G$, then for each $g\in G$ there is $f\in G$ such that $g\in f\cdot B$ and there is $h\in H\cap (f\cdot B)$.
Since $f^{-1}\cdot g\in B$ and $f^{-1}\cdot h\in B$, so 
$h^{-1}\cdot g\in B\ast B^{-1}$ and $g\in H\cdot (B\ast B^{-1})$.
Thus $G=(B\ast B^{-1})\cdot H$ and $H$ is of bounded index in $G$.

\end{proof}

\section{Connectivity in large scale groups}

In this section we introduce concepts needed to generalize being finitely generated to being boundedly generated.

\begin{Definition}\label{KchainDef}
Suppose $K$ is a symmetric subset of a group $G$  (that means $K^{-1}=K$). A \textbf{$K$-chain} is a finite sequence $g_1,\ldots, g_k$ of elements of $G$ such that $g_i^{-1}\cdot g_{i+1}\in K$ for each $i < k$. A subset $C$ of $G$ is \textbf{$K$-connected} if every two elements of $C$ can be connected by a $K$-chain. $C$ is a \textbf{$K$-component} of $A\subset G$ if $C$ is an equivalence class of the equivalence relation $\sim$ on $A$ defined as follows: $g\sim h$ if $g$ and $h$ can be connected by a $K$-chain in $A$. We always assume $1_G\in K$ as that does not change connectivity.

If $G$ is $K$-connected, then the \textbf{$K$-norm} on $G$ is the length of the shortest $K$-chain joining $1_G$ and $g\in G$.
\end{Definition}

\begin{Definition}
A large scale group $G$ is \textbf{boundedly generated} if there is a symmetric bounded set $K$ such that every element $g$ of $G$ is a finite product of elements of $K$. Equivalently, $G$ is $K$-connected. In this case we say $G$ is \textbf{$K$-generated}.
\end{Definition}

\begin{Proposition}
Suppose $K$ is a symmetric bounded subset of a group large scale $G$. If, for some bounded subset $B$ of $G$, $G\setminus B$ has has finitely many unbounded $K$-components and the union of all bounded $K$-components is bounded, then $G$ is boundedly generated.
\end{Proposition}
\begin{proof}
Let $L$ be the union of $K\cup B$ and of the following:\\
1. The union of all bounded $K$-components of $G\setminus B$,\\
2. One point from each non-empty unbounded $K$-component of $G\setminus B$.\\
Put $M=L\cup L^{-1}$ and notice $G$ is $M$-connected.
Indeed, the $M$-component of $1_G$ contains $B$ and all $K$-components of $G\setminus B$.
\end{proof}

\begin{Proposition}
Suppose $K$ is a symmetric bounded subset of a large scale group $G$
that is $K$-connected. If every bounded subset $B$ of $G$ can be covered by finitely many sets of the form $g\cdot K$, then for every bounded subset $L$ of $G$ its complement $G\setminus L$ has finitely many $K^4$-components.
\end{Proposition}
\begin{proof}
Let $\mathcal{U}:=\{g\cdot K\}_{g\in G}$ and $M:=st(L,\mathcal{U})$.
Choose $g_i\in G$, $i\leq m$, such that $M\subset \bigcup\limits_{i=1}^m g_i\cdot K$.
Pick $a\in L$ (if $L=\emptyset$, then $G\setminus L$ has exactly one $K$-component) and for each $x\in G\setminus L$ choose a $K$-chain $c_x$ joining
$x$ to $a$. Let $i(x)\leq m$ be the first index of $ \bigcup\limits_{i=1}^m g_i\cdot K$ encountered by $c_x$ and let $l(x)$ be link of $c_x$ preceding 
meeting of $ \bigcup\limits_{i=1}^m g_i\cdot K$ or $l(x)=x$ if $x\in \bigcup\limits_{i=1}^m g_i\cdot K$.
If $x,y\in G\setminus L$ have the same index $i(x)=i(y)$, then $l(x)\cdot l(y)^{-1}\in K^4$ using the chain $l(x)\to g_{i(x)}\cdot K\to l(y)$. That means $G\setminus L$ has at most $m$ $K^4$-components.
\end{proof}

\begin{Corollary}
Suppose $K$ is a symmetric neighborhood of $1_G$ in a locally compact topological group $G$ such that $G=<K>$. For every bounded subset $L$ of $G$ its complement $G\setminus L$ has finitely many $K^4$-components.
\end{Corollary}
\begin{proof}
The closure $cl(L)$ is compact, so it can be covered by finitely many sets of the form $g\cdot K$, $g\in G$.
\end{proof}

\begin{Proposition}\label{CyclicSubgroupProp}
Suppose $K$ is a symmetric bounded subset of a large scale group $G$ containing $1_G$, $B$ is a bounded subset of $G$ so that $G\setminus B$ has $2$ $K$-components $L$ and $R$ on which $G$ acts trivially, and both $L$ and $R$ are unbounded. $G$ has a cyclic subgroup of bounded index provided one of the following conditions is satisfied:\\
1. $B$ is $K$-connected.\\
2. $B\subset h\cdot K^n$ for some $h\in G$ and $n\ge 1$.
\end{Proposition}
\begin{proof}
$G$ \textbf{acts trivially} on a subset $A$ if the symmetric difference $A\Delta (g\cdot A)$ is bounded for each $g\in G$.

Choose $a_L\in L$ and $a_R\in R$. Put $g:=a_L\cdot a_R^{-1}$.
Since $g^{-1}\cdot R\Delta R$ is bounded, there is $x\in R$ such that $g^{-1}\cdot x\in R$. Choose a $K$-chain $c$ in $R$ joining $a_R$ and $g^{-1}\cdot x$. Notice $g\cdot c$ is a $K$-chain joining $a_L$ and $x$. Therefore $G$ is $K$-connected and $B\ne\emptyset$. By switching to $c^{-1}\cdot B$,
$c^{-1}\cdot L$, and $c^{-1}\cdot R$ for some $c\in B$, we may assume $1_G\in B$.

In the case of Condition 2 we have a prescribed $n$. If it is not satisfied (i.e. Condition 1 holds), we put $n=1$.
If $g\in G\setminus (B\cdot B^{-1}\cup B(B,n+1))$, then $g\cdot B$ is disjoint from 
$B$, hence it must be contained either in $R$ or in $L$. Indeed, it is so if $B$ is $K$-connected. If it is not $K$-connected, and, say $g\in R$, then $g\cdot B$ cannot intersect $L$ as in such a case there is a $K$-chain joining $L$ to $g$ of length at most $n$, so it must intersect $B$ resulting in $g\in B(B,n+1)$, a contradiction.
 Thus $g\cdot B\subset R$ if $g\in R\setminus B\cdot B^{-1}$
and $g\cdot B\subset L$ if $g\in L\setminus B\cdot B^{-1}$

Choose $g\in R\setminus (B\cdot B^{-1}\cup B(B,n+1))$. 
Since $g\cdot B\subset R$, $L$ must be contained in $g\cdot L$. Otherwise $L\subset g\cdot R$ and $(g^{-1}\cdot L)\Delta L$ is unbounded (as it contains $L$), a contradiction.
Now, we need $g\cdot R\subset R$. It is so if $g^{-1}\cdot B\subset L$
as that implies $B\subset g\cdot L$, so $g\cdot R\subset G\setminus g\cdot L\subset G\setminus (L\cup B)=R$.
So assume $g^{-1}\cdot B\subset R$. Now, $B\subset g\cdot R$,
so $g\cdot L=L$ as $g\cdot L$ misses $B$ and is $K$-connected.
There exist elements $c_L\in L$ and $c_B\in B$ such that $c_L=c_B\cdot k$ for some $k\in K$. Now, $g\cdot c_B\cdot k=g\cdot c_L\in L$
and $ g\cdot c_B\cdot k\in (g\cdot B)\cdot k$ which means we can get from $L$ to $g\cdot B$ via a $K$-chain bypassing $B$, a contradiction.

By induction we get $g^k\cdot L\subset g^{k+1}\cdot L$ and $ g^{k+1}\cdot R\subset g^{k}\cdot R$ for all integers $k$. To complete the proof it suffices to show that the union of all sets $g^k\cdot (B\cup (g\cdot L\setminus L))$ equals $G$ as $B\cup (g\cdot L\setminus L)$ is bounded.
Given $x\in G$ find $m$ that minimizes all distances $dist(x,g^k\cdot B)$, $k\in \mathbb{Z}$, as measured via the $K$-norm. Of interest is the case of that minimum being positive.
In that case either $x\in g^m\cdot R$ or $x\in g^m\cdot L$.
In the first case $x$ cannot be in $g^{m+1}\cdot R$ as then any $K$-chain joining $x$ to $ g^m\cdot B$ passes through $ g^{m+1}\cdot B$, a contradiction. Thus $x\in g^m\cdot R\setminus g^{m+1}\cdot R\subset
(g^{m+1}\cdot L\setminus g^{m}\cdot L)\cup g^{m+1}\cdot B$.
In the second case $x$ cannot be in $g^{m-1}\cdot L$ as then any $K$-chain joining $x$ to $ g^m\cdot B$ passes through $ g^{m-1}\cdot B$, a contradiction. Thus $x\in g^m\cdot L\setminus g^{m-1}\cdot L$.
\end{proof}

\section{Svarc-Milnor Lemma for large scale groups}

Geometric group theorists traditionally restrict their attention to finitely
generated groups equipped with a word metric. A typical proof of \v Svarc-Milnor Lemma
(see \cite{Roe lectures} or \cite{BH}, p.140) involves such metrics.

 \begin{Theorem}
A group
$G$ acting properly and cocompactly via isometries on a length space $X$
is finitely generated and
induces a quasi-isometry equivalence $g\to g\cdot x_0$ for any $x_0\in X$.
\end{Theorem}

\begin{Theorem}\cite{BDM1} 
If a group $G$ acts cocompactly and properly via isometries on a proper metric space $X$, then $g\to g\cdot x_0$ induces a coarse equivalence between $G$ and $X$
for all $x_0\in X$.
\end{Theorem}

\begin{Definition}
A group $G$ acts on a large scale space $X$ by \textbf{uniform coarse equivalences} if for every uniformly bounded cover $\{U_s\}_{s\in S}$ of $X$ the cover $\{g\cdot U_s\}_{s\in S, g\in G}$ is uniformly bounded.

The action is \textbf{cobounded} if there is a bounded subset $B$ of $X$ so that $G\cdot B=X$. 

If $G$ is a large scale group, then the action is \textbf{proper} if for every bounded subset $B$ of $X$ the set $\{g\in G | (g\cdot B)\cap B\ne\emptyset\}$ is bounded in $G$ and for every bounded subset $K$ of $G$ the set $K\cdot x$ is bounded in $X$ for each $x\in X$.
\end{Definition}

\begin{Theorem} \label{Svarc-MilnorForLSGroups}
Suppose a large scale group $G$ acts by uniform coarse equivalences on a large scale space $X$. If the action is proper and cobounded, then for each $x_0\in X$ the map $g\to g\cdot x_0$ is a coarse equivalence.
\end{Theorem}
\begin{proof}
The map $g\to g\cdot x_0$ is large scale continuous as for each bounded subset $K$ of $G$ the family $\{g\cdot K\}_{g\in G}$
is sent to the family $\{g\cdot (K\cdot x_0)\}_{g\in G}$ which is uniformly bounded as $K\cdot x_0$ is bounded in $X$.

Given a uniformly bounded cover $\mathcal{U}=\{U_s\}_{s\in S}$ of $X$ let $B:=st(x_0,g\cdot \mathcal{U})$. Notice $B$ is bounded.
The set $K:=\{g\in G | (g\cdot B)\cap B\ne\emptyset\}$ is bounded in $G$.
Now, given $s\in S$ such that $g_s\cdot x_0\in U_s$ for some $g_s\in G$, then for any $g\in G$ so that
$g\cdot x_0\in U_s$ one has $x_0\in g^{-1}\cdot U_s$. Therefore
$ g^{-1}\cdot U_s\subset B$ resulting in $x_0\in ((g_s^{-1}\cdot g)\cdot B)\cap B$. Hence $g\in g_s\cdot K$ and the inverse of the cover $U$ under the map
$g\to g\cdot x_0$ is a refinement of the cover $\{g\cdot K\}_{g\in G}$.
Thus the map is a coarse embedding.

If $G\cdot B=X$ for some bounded $B$ in $X$,
then $st(G\cdot x_0, \{g\cdot B\}_{g\in G})=X$, so the inclusion $G\cdot x_0\to X$ is a coarse equivalence.
\end{proof}

\section{Metrizable large scale groups}

In this section we discuss large scale groups whose large scale structure is metrizable or coarsely equivalent to a geodesic space.

\begin{Proposition}\label{MetrizabilityOfLSG}
Suppose $G$ is a large scale group. The following conditions are equivalent:\\
1. $G$ is metrizable (i.e. its large scale structure is generated by a metric).\\
2. The bornology of $G$ has a countable basis.\\
3. There is a left-invariant metric $d$ on $G$ inducing the large scale structure on $G$.
\end{Proposition}
\begin{proof}
1)$\implies$2) and 3)$\implies$1) are obvious.\\
2)$\implies$3) Choose an increasing sequence $\{B_n\}_{n\ge 1}$ of symmetric bounded subsets of $G$ containing $1_G$ that serves as a basis of the bornology of $G$. We may assume $B_n\ast B_n\subset B_{n+1}$ for each $n\ge 1$. Define the norm $|g|$ on $G$ as follows:\\
1. $|1_G|=0$.\\
2. If $g\ne 1_G$, then $|g|$ is the smallest $n$ such that $g\in B_n$.

Define the the metric $d$ on $G$ via $d(g,h)=|g^{-1}\cdot h|$. 
Notice the cover $\{B(g,n+1)\}_{g\in G}$ equals $\{g\cdot B_n\}_{g\in G}$,
so the large scale structure of $G$ equals the large scale structure induced from $(G,d)$.
\end{proof}

\begin{Proposition}\label{CayleyGraph}
Suppose $G$ is a large scale group. The following conditions are equivalent:\\
1. $G$ is generated by a symmetric bounded set $K$ and its bornology has $\{K^n\}_{n\ge 1}$ as a countable basis.\\
2. $G$ is coarsely equivalent to a connected graph $\Gamma$ whose set of vertices is equal to $G$ and the graph metric is left-invariant.\\
3. $G$ is coarsely dominated by a geodesic space, i.e. there is a geodesic space $X$ and large scale continuous functions $f:G\to X$, $g:X\to G$ such that $g\circ f$ is close to $id_G$.
\end{Proposition}
\begin{proof}
1)$\implies$2). Extend $G$ to a connected graph $\Gamma$ by requiring that $g$ and $h$ form an edge if and only if $g^{-1}\cdot h\in K$.
Let $d$ be the graph metric on $\Gamma$.
Notice the cover $\{B(g,n+1)\}_{g\in G}$ equals $\{g\cdot K^n\}_{g\in G}$,
so the large scale structure of $G$ equals the large scale structure induced from $(G,d)$. \\
2)$\implies$3) is obvious.\\
3)$\implies$1). Let $\alpha:G\to X$ and $\beta:X\to G$ be two large scale continuous functions such that $\beta\circ\alpha$ is close to $id_G$.
Choose a symmetric bounded subset $K$ of $G$ such that
$ \beta\circ\alpha(g)\in g\cdot K$ for each $g\in G$ and the cover
$\{\beta(B(x,2)\}_{x\in X}$ is a refinement of the cover $\{g\cdot K\}_{g\in G}$.
Given a bounded subset $L$ of $G$ containing $1_G$ there is $m > 0$
such that $\alpha(L)\subset B(x_0,m)$, $x_0=\alpha(1_G)$.
If $g\in L$, then one can connect $x_0$ to $\alpha(g)$ via a chain
$x_0,\ldots, x_n$ so that $x_{i+1}\in B(x_i,1)$ for each $i < n$.
That means $\beta(x_{i+1})$ and $\beta(x_{i})$ belong to a set of the form $h\cdot K$ resulting in $\beta(x_{i+1})\cdot \beta(x_{i})^{-1}$ belonging to
$K\cdot K$. That implies $\beta(\alpha(g))\in (K\cdot K)^n$.
Hence $ \beta(\alpha(g))\in (K\cdot K)^n\cap (g\cdot K)$ and $g\in K^{2n+1}$.
As $n$ can be chosen uniformly for all $g\in L$, $L\subset K^{2n+1}$.
\end{proof}

\begin{Definition}\label{CayleyGraphDef}
Let $G$ be a coarsely geodesic large scale group. A \textbf{Cayley graph of $G$} is a connected graph $\Gamma$ whose set of vertices equals $G$ and such that the inclusion $G\to \Gamma$ is a coarse equivalence, where $\Gamma$ is equipped with its graph metric which is left-invariant. The metric on $G$ induced from a Cayley graph will be called a \textbf{Cayley metric}.
\end{Definition}

\begin{Observation}
The set $K$ of all vertices of $G$ at distance $1$ from $1_G$ is symmetric and generates $G$.
\end{Observation}

\begin{Lemma}\label{CharOfCClopenInLSGroups}
Suppose $G$ is a large scale group, $B$ containing $1_G$ is a bounded subset of $G$, and $\mathcal{U}=\{g\cdot B\}_{g\in G}$. If $A\subset G$, then 
$A\cdot B\subset st(A,\mathcal{U})\subset A\cdot B^{-1}\cdot B$.
\end{Lemma}
\begin{proof}
If $a\in A$ and $b\in B$, then $a,a\cdot b\in a\cdot B$ so $ A\cdot B\subset st(A,U)$. If $a\in A\cap (g\cdot B)$, then $g\in a\cdot B^{-1}$
and $g\cdot B\subset A\cdot B^{-1}\cdot B$.
\end{proof}

\begin{Lemma}\label{CClopenStructure}
A subset $A$ of a large scale group is coarsely clopen if and only if for each bounded subset $B$ of $G$ the set $(A\cdot B)\cap (A^c\cdot B)$ is bounded.
\end{Lemma}
\begin{proof}
Apply \ref{CharOfCClopenInLSGroups}.
\end{proof}

\begin{Proposition}\label{ComponentsOfNCCInCoarselyGeodesic}
If $A$ is a coarsely clopen subset of $G$ equipped with a Cayley metric, then every union of $1$-components of $A$ is coarsely clopen.
Moreover, there is a bounded subset $B$ of $G$ such that every $1$-component of $G\setminus B$ intersecting $A\setminus B$ is contained in $A\setminus B$.
\end{Proposition}
\begin{proof}
Let $K=\{g\in G | d(g,1_G)\leq 1\}$.
$C\subset A$ is a union of $1$-components of $A$ if and only if $C\cdot K=C$.
Therefore $C\cdot K^n=C$ for all $n\ge 1$. Now,
$(C\cdot K^n)\cap (C^c\cdot K^n)\subset ((C\cdot K^n)\cap (A\setminus C)\cdot K^n))\cup (C\cdot K^n)\cap (A^c\cdot K^n)=(C\cdot K^n)\cap (A^c\cdot K^n)$, as $(C\cdot K^n)\cap (A\setminus C)\cdot K^n)=C\cap (A\setminus C)=\emptyset$. Thus, $(C\cdot K^n)\cap (C^c\cdot K^n)$ is bounded for all $n\ge 1$ resulting in $(C\cdot B)\cap (C^c\cdot B)$ being bounded for all bounded $B$.

Let $B:=(A\cdot K)\cap (A^c\cdot K)$. If $g\in A\setminus B$ and $g\cdot k\in A^c$ for some $k\in K$, then
implies $g=(g\cdot k)\cdot k^{-1}\in A^c\cdot K$, hence $g\in B$, a contradiction. That means any $K$-chain in $X\setminus B$ starting from $g\in A\setminus B$ must remain in $A\setminus B$.
\end{proof}

\begin{Corollary}\label{MetrizabilityOfEndsOfLSGroup}
Suppose $G$ is a large scale group metrizable by a Cayley metric $d$.
$Ends(G)$ is metrizable if both of the following conditions are satisfied:\\
1. For each bounded subset $B$ of $G$ the union of bounded $1$-components of $G\setminus B$ is bounded.\\
2. For each bounded subset $B$ of $G$ its complement $G\setminus B$ has finitely many unbounded $1$-components.\\
Moreover, $Ends(G)$ can be described as the family of decreasing sequences $\{A_i\}_{i\ge 1}$
 of unbounded $1$-components of $G\setminus K_i$, where $\{K_i\}_{i\ge 1}$ is an increasing sequence of bounded subsets of $G$ that is a basis of bounded subsets of $G$.
\end{Corollary}
\begin{proof}
Suppose $\{K_i\}_{i\ge 1}$ is an increasing sequence of bounded subsets of $G$ that is a basis of bounded subsets of $G$.
By \ref{ComponentsOfNCCInCoarselyGeodesic} every unbounded $1$-component $A_i$ of $G\setminus K_i$ is coarsely clopen. Therefore each decreasing sequence $\{A_i\}_{i\ge 1}$
 of unbounded $1$-components of $G\setminus K_i$ is contained in an end $E$ of $G$. It cannot be contained in two different ends $E$ and $E'$. Indeed, in that case we can pick disjoint coarsely clopen subsets $C\in E$ and $D\in E'$. 
By \ref{ComponentsOfNCCInCoarselyGeodesic} there are bounded subsets $B_C$ and $B_D$ of $G$ such that $C\setminus B_C$ is a union of $1$-components of $G\setminus B_C$ and $D\setminus B_D$ is a union of $1$-components of $G\setminus B_D$. Find $j\ge 1$ such that $B_C\cup B_D\subset K_j$. Notice $A_j$ does not intersect at least one of $C\setminus B_C$ or $D\setminus B_D$, a contradiction.

Suppose $E\in Ends(G)$. By \ref{ComponentsOfNCCInCoarselyGeodesic} there is a bounded subset $B$ of $G$ such that $A\setminus B$ is a union of $1$-components of $G\setminus B$. Find $j\ge 1$ such that $B\subset K_j$. 
Now, $A\setminus K_j$ is the union of $1$-components of $G\setminus K_j$,
so $A$ must contain exactly one unbounded $1$-component of $G\setminus K_j$. That shows $Ends(G)$ has a countable basis, hence it is metrizable.
\end{proof}

\begin{Lemma}\label{NonMetrizableEndsOfLSGroups}
Suppose $G$ is a large scale group metrizable by a Cayley metric $d$ and
$\{A_i\}_{i\ge 1}$ is a family of mutually disjoint coarsely clopen subsets of $G$. If for each infinite subset $P$ of naturals $\mathbb{N}$ the union $\bigcup\limits_{i\in P}A_i$ is an unbounded coarsely clopen subset of $G$, then the space of ends $Ends(G)$ of $G$ is not metrizable.
\end{Lemma}
\begin{proof}
Given an ultrafilter $\mathcal{F}$ of $\mathbb{N}$ consisting of infinite sets
choose $x({\mathcal{F}})\in \bigcap\limits_{P\in \mathcal{F}}Ends(\bigcup\limits_{i\in P}A_i)$.
Given two different ultrafilters $\mathcal{F}_1$ and $\mathcal{F}_2$
there are disjoint $P\in \mathcal{F}_1$ and $Q\in \mathcal{F}_2$ resulting in $x(\mathcal{F}_1)\ne x(\mathcal{F}_2)$. That means $Ends(G)$ contains at least $2^{c}$ points, where $c$ is the cardinality of reals, so it cannot be metrizable.
\end{proof}

\begin{Corollary}
Suppose $G$ is a large scale group metrizable by a Cayley metric $d$.
$Ends(G)$ is non-metrizable if and only if one of the following conditions is satisfied:\\
1. There is a bounded subset $B$ of $G$ such that the union of bounded $1$-components of $G\setminus B$ is unbounded.\\
2. There is a bounded subset $B$ of $G$ such that of $G\setminus B$ has infinitely many unbounded $1$-components.
\end{Corollary}
\begin{proof}
In case of 1) choose a sequence $C_n$ of bounded $1$-components of $G\setminus B$ such that $g_n\in C_n$ and $d(g_n,1_G)\to\infty$.
In case of 2) choose a sequence $C_n$ of unbounded $1$-components of $G\setminus B$ that are mutually disjoint.
Apply \ref{ComponentsOfNCCInCoarselyGeodesic} and \ref{NonMetrizableEndsOfLSGroups}.
\end{proof}

\section{Large scale groups of bounded geometry}

In this section we introduce the concept of bounded geometry for large scale groups. Typically, bounded geometry is defined for metric spaces by requiring that for each $r > 0$ there is $N_r\in \mathbb{N}$ such that every $r$-ball contains at most $N_r$ elements. We want a coarse invariant, so we extend this definition to arbitrary large scale spaces as follows:

\begin{Definition}
A large scale space $X$ has \textbf{bounded geometry} if it is coarsely equivalent to a large scale space $Y$ with the property that for each uniformly bounded cover $\mathcal{U}$ of $Y$ there is $N(\mathcal{U})\in \mathbb{N}$ such that every element of $\mathcal{U}$ contains at most $N(\mathcal{U})$ elements.
\end{Definition}

\begin{Corollary}
Suppose $X$ and $Y$ are large scale spaces and $Y$ has bounded geometry. If there is a coarse embedding $f:X\to Y$, then $X$ has bounded geometry.
\end{Corollary}
\begin{proof}
Pick a coarse equivalence $g:Y\to Z$, where $Z$ has the property that for each uniformly bounded cover $\mathcal{U}$ of $Z$ there is $N(\mathcal{U})\in \mathbb{N}$ such that every element of $\mathcal{U}$ contains at most $N(\mathcal{U})$ elements. Observe that $X$ is coarsely equivalent to $f(g(Y))$.
\end{proof}

\begin{Proposition}
A large scale group $G$ is of bounded geometry if and only if there is a  bounded set $K$ such that for every bounded subset $B$ of $G$ there are elements $g_i\in G$, $i\leq k$, so that $B\subset \bigcup\limits_{i=1}^k g_i\cdot K$.
\end{Proposition}
 \begin{proof}
Suppose $f:G\to Y$ is a coarse embedding, where $Y$ has the property that for each uniformly bounded cover $\mathcal{U}$ of $Y$ there is $N(\mathcal{U})\in \mathbb{N}$ such that every element of $\mathcal{U}$ contains at most $N(\mathcal{U})$ elements. The family $\{f^{-1}(y)\}_{y\in Y}$ is a uniformly bounded cover of $G$, so there is a bounded subset $K$ of $G$ with the property that for every $y\in Y$ there is $g(y)\in G$ such that
$f^{-1}(y)\subset g(y)\cdot K$. Given a bounded subset $B$ of $G$, $f(B)$ is bounded, hence it is finite and $B\subset \bigcup\limits_{y\in f(B)} f^{-1}(y)\subset \bigcup\limits_{y\in f(B)} g(y)\cdot K$.

Suppose there is a  bounded set $K$ such that for every bounded subset $B$ of $G$ there are elements $g_i\in G$, $i\leq k$, so that $B\subset \bigcup\limits_{i=1}^k g_i\cdot K$. We may assume $K$ is symmetric by switching to $K\cup K^{-1}$. Also, add $1_G$ to $K$. Choose a maximal subset $Y$ of $G$ with the property that $g\ne h\in Y$ implies $g^{-1}\cdot h\notin K\cdot K^{-1}$. 

Notice that $Y$ is of bounded geometry. Indeed, given a uniformly bounded cover $\mathcal{U}$ of $Y$ there is a bounded subset $B$ of $G$ such that $U$ is a refinement of $\{g\cdot B\}_{g\in G}$. Suppose $B\subset \bigcup\limits_{i=1}^k g_i\cdot K$ and $y,z\in Y\cap (g\cdot B)$ are two different elements of $Y$. If $y,z\in g\cdot g_i\cdot K$, then $y^{-1}\cdot z\in K^{-1}\cdot K$, a contradiction.

To conclude that the inclusion $Y\to G$ is a coarse equivalence, it is sufficient to show $G= \bigcup\limits_{y\in Y} y\cdot K$.
Suppose $x\in G\setminus Y$. There is $y\in Y$ such that $k:=x^{-1}\cdot y\in K$. Now, $x=y\cdot k^{-1}$ and we are done.
\end{proof}

John Roe \cite{RoeAsDimHyp} defines a geodesic metric space $X$ to be of \textbf{bounded growth} if for each $s > 0$ there is a number $N_s$ such that each ball of radius $S+s$ in $X$ can be covered by at most $N_s$ balls of radius $S$. We consider this definition excessively restrictive as it is of an all-scale character instead of being of a large scale character. It would seem that changing Roe's definition to require that there is $G > 0$ such that
for each $s > G$ there are numbers $N_s$ and $D_s$ such that each ball of radius $r+s$, $r\ge D_s$, in $X$ can be covered by at most $N_s$ balls of radius $r$, would be an improvement but it is not clear if the new concept is an invariant of quasi-isometries. 

Our next result shows that metrizable large scale groups of bounded geometry have a property resembling bounded growth.

\begin{Proposition}\label{BasicBdPropInGroups}
If a large scale group $G$ is metrizable by the $K$-norm, then the following conditions are equivalent:\\
1. For every bounded subset $B$ of $G$ there are elements $g_i\in G$, $i\leq k$, so that $B\subset \bigcup\limits_{i=1}^k g_i\cdot K$.\\
2. There are elements $g_i\in G$, $i\leq k$, so that $K\ast K\subset \bigcup\limits_{i=1}^k g_i\cdot K$.\\
3. For each $s\ge 1$ there is $N_s\ge 1$ such that for each $n\ge 1$ the set $K^{n+s}$ can be covered by at most $N_s$ sets of the form $g\cdot K^n$.
\end{Proposition}
\begin{proof}
1)$\implies$2) and 3)$\implies$2) are trivial.

2)$\implies$1). Notice $K\ast K\ast K\subset \bigcup\limits_{i=1}^k g_i\cdot K\ast K\subset \bigcup\limits_{i,j=1}^k g_i\cdot g_j\cdot K$. Apply induction to get that for every $n\ge 1$ there are elements $h_i\in G$, $i\leq k$, so that $K^n\subset \bigcup\limits_{i=1}^k h_i\cdot K$. Since every bounded subset $B$ of $G$ is contained in some $K^n$, 1) follows.

2)$\implies$3). As in the above there is $N_s\ge 1$ such that $K^{s+1}$ can be covered by at most $N_s$ sets of the form $g\cdot K$.
Therefore $K^{n+s}=K^{s+1}\ast K^{n-1}$ can be covered by at most $N_s$ sets of the form $g\cdot K\ast K^{n-1}$.
\end{proof}

\begin{Corollary}
The following large scale groups are of bounded geometry:\\
1. Groups with bornology consisting of finite sets,\\
2. Locally compact topological groups.
\end{Corollary}
\begin{proof}
In case 2) any pre-compact neighborhood $K$ of $1_G$ works.
\end{proof}

\begin{Observation}
Notice that reals form a large scale group of bounded geometry but the Cayley graph of reals constructed using $K=[-1,1]$ is not of bounded growth in the sense of Roe's definition. Indeed, for $s < 1/2$ the ball $B(r,2s)$ at any vertex $r$ of the graph cannot be covered by finitely many balls of radius $s$.
\end{Observation}

\begin{Definition}\label{BasisForBoundedlygeneratedSubgroupsDef}
A sequence of subgroups $\{G_i\}_{i\ge 1}$ of a large scale group $G$ is a \textbf{basis for boundedly generated subgroups} of $G$ if every boundedly generated subgroup $H$ of $G$ is contained in some $G_i$. Equivalently, for any bounded subset $B$ of $G$ there is $i\ge 1$ such that $B\subset G_i$.
\end{Definition}

\begin{Proposition}\label{BornologyHasACountableBasis}
1. If $G$ is a countable union of its bounded subsets and has a bounded geometry, then its bornology has a countable basis. Consequently, $G$ has a countable basis of boundedly generated subgroups.\\
2. If $G$ is boundedly generated and has a bounded geometry, then it is coarsely geodesic.
\end{Proposition}
\begin{proof}
1. Pick a bounded set $K$ such that for every bounded subset $B$ of $G$ there are elements $g_i\in G$, $i\leq k$, so that $B\subset \bigcup\limits_{i=1}^k g_i\cdot K$. Suppose $G= \bigcup\limits_{i=1}^\infty B_i$, where each $B_i$ is bounded and $B_i\subset B_j$ if $i < j$. Each $g\in G$ has an index $n(g)$ such that $g\in B_{n(g)}$. Now, if $B$ is bounded and $B\subset \bigcup\limits_{i=1}^k g_i\cdot K$, then for $m\ge n(g_i)$ for all $i\leq n$ one has
$B\subset B_m\cdot K$.

2. Pick a bounded set $K$ such that for every bounded subset $B$ of $G$ there are elements $g_i\in G$, $i\leq k$, so that $B\subset \bigcup\limits_{i=1}^k g_i\cdot K$. We may assume $K$ generates $G$. Each $g\in G$ has an index $n(g)$ such that $g\in K^{n(g)}$. Now, if $B$ is bounded and $B\subset \bigcup\limits_{i=1}^k g_i\cdot K$, then for $m\ge n(g_i)$ for all $i\leq n$ one has
$B\subset K^{m+1}$. By \ref{CayleyGraph}, $G$ is coarsely geodesic.
\end{proof}

\begin{Proposition}\label{FinitelyManyComponentsBdGrowth}
Suppose $G$ is a large scale group that has a bounded symmetric subset $K$ such that $\{K^n\}_{n\ge 1}$ is a basis for the bornology of $G$
and $K^2$ can be covered by $m < \infty$ sets of the form $g\cdot K$, $g\in G$. 
If $B$ is a bounded subset of $G$ equipped with the Cayley metric induced by $K^4$, then $G\setminus B$ has only finitely many $1$-components.
\end{Proposition}
\begin{proof}
Assume $1_G\in B$ and for every $g\in G$ pick a geodesic $K$ chain $c(g)$ from $g$ to $1_G$. Choose $g_i\in G$, $i\leq p$, such that
$B\cdot K\subset\bigcup\limits_{i=1}^p g_i\cdot K$. Given $g\in G\setminus B$, the chain $c(g)$ has the first term $f(g)$ in $B$ and that term must land
in some $g_j\cdot K$. Assign one such index $j\leq p$ to $g$ thus creating a function $i(g)$ from $G\setminus B$ to natural numbers at most $p$. If $i(g_1)=i(g_2)$, then one can jump from the previous element of $c(g_1)$ to $l(g_1)$ to the previous element of $c(g_2)$ to $l(g_2)$ via an element of $K^4$. That means $g_1$ and $g_2$ can be connected via a $K^4$-chain outside of $B$ and are in the same $1$-component of $G\setminus B$.
\end{proof}

\begin{Corollary}
The space of ends of a coarsely geodesic large scale group of bounded geometry is metrizable.
\end{Corollary}
\begin{proof}
Apply \ref{MetrizabilityOfEndsOfLSGroup} and \ref{FinitelyManyComponentsBdGrowth}.
\end{proof}

\subsection{Coarse hyperbolicity}

\begin{Definition}\label{CoarselyHypGroupDef}
A large scale group $G$ is \textbf{coarsely hyperbolic} if it is large scale equivalent to a geodesic space that is hyperbolic in the sense of Gromov.
\end{Definition}

Notice that $G$ has Cayley graphs if it is coarsely hyperbolic and, since being hyperbolic is a coarse invariant of geodesic spaces, $G$ is coarsely hyperbolic if and only if one (hence every) of its Cayley graphs is hyperbolic. Therefore, if $G$ is metrizable via a $K$-norm, it is coarsely hyperbolic if there is $\delta > 0$ such that
for every two $K$-geodesics $c$ and $d$ on $G$ emanating from $1_G$ (those are $K$-chains of length equal to the $K$-norm of the terminal elements) the distance between $i$th elements of the chains is less than $\delta$ if $i\leq (|c|+|d|-d(g_c,g_d)/2$, where $g_c$ is the terminal point of $c$ and $g_d$ is the terminal point of $d$.

The proof of the theorem below is a simplification of the one in \cite{RoeAsDimHyp}. Also, we fix a gap in the original proof by showing that any ball $B(x,ps)$ intersects at most $N_{2\delta}$ elements of the cover of $A$, not only those with $x\in A$.
\begin{Theorem}\label{AsdimOfHypGroups}
Suppose $G$ is a large scale group of bounded geometry. If $G$ is coarsely hyperbolic, then the asymptotic dimension of $G$ is finite.
\end{Theorem}
\begin{proof}
If $G$ is bounded, it is of asymptotic dimension at most $0$, so assume $G$ is unbounded.

Choose $K$ satisfying two conditions:\\
1. $G$ is metrizable via the $K$-norm,\\
2. There is $m\ge 1$ such that $K\ast K$ is covered by $m$ balls of radius $2$ in $G$, i.e. sets of the form $g\cdot K$, $g\in G$.\\
3. There is $\delta > 1$ such that
for every two $K$-geodesics $c$ and $d$ on $G$ emanating from $1_G$  the distance between $i$th elements of the chains is less than $\delta$ if $i\leq (|c|+|d|-d(g_c,g_d)/2$, where $g_c$ is the terminal point of $c$ and $g_d$ is the terminal point of $d$.

By \ref{BasicBdPropInGroups} for each $t\ge 1$ there is $N_t$ such that for each $S \ge 1$ the set $K^{S+t}$ can be covered by at most $N_t$ sets of the form $g\cdot K^S$.

Consider $A:=K^{2n+2p\cdot s}\setminus K^{2n}$ for some $n,p\ge 1$ so that $n > ps$.
Notice $K^{2n}\setminus K^{2n-1}\ne\emptyset$ as otherwise $K^q=K^{2n}$ for all $g\ge 2n$ and $G$ is bounded. Similarly, $A\ne\emptyset$. In $K^{2n}\setminus K^{2n-1}$ choose a maximal set of points $\{x_i\}_{i\in J}$ that are separated by at least $2p\cdot s$ in terms of the $K$-metric. For each $i\in J$ choose a $K$-geodesic $c_i$ from $1_G$ to $x_i$.

For each $g\in G$ choose a $K$-geodesic $c_g$ from $1_G$ to $g$.
Given $g\in A$ let $r(g)$ be the last term of $c_g$ belonging to $K^{2n}\setminus K^{2n-1}$. Let $U_i$, $i\in J$, be the set of all $g\in A$ so that 
$d(r(g),x_i)\leq 2ps$. Obviously, $U_i$, $i\in J$, form a cover of $A$.
If $g\in U_i$, then $d(g,x_i)\leq d(g,r(g))+d(r(g),x_i)< 2ps+2ps=
4ps$.
 That means each $U_i$ is of diameter at most $8ps$.

Suppose $x\in G$ and $B(x,ps)$ intersects $U_i$ for some $i\in J$ and $y\in B(x,ps)\cap U_i$. Since $|x|\ge |y|-ps > 2n-ps$, $z:=c_x(2n-ps)$ exists.
Moreover, $d(z,c_y(2n-ps)) < \delta$ as 
$|x|+|y|-d(x,y)\ge 2n-ps+2n-ps=2(2n-ps)$. Also,
$|r(y)|+|x_i|-d(r(y),x_i)\ge 2n+2n-2ps=2(2n-ps)$, so
$d(c_y(2n-ps),c_i(2n-ps)) < \delta$. Thus $d(z,c_i(2n-ps)) < 2\delta$
and $d(z,x_i) < ps+2\delta$.
However, $B(z,ps+2\delta)$ can be covered by $N_{2\delta}$ balls of radius $ps$ and each of them can contain at most one $x_i$.

Now we are ready to show that the asymptotic dimension of $G$ is at most $2\cdot N_{2\delta}-1$. Indeed, given $r > 0$ we choose $p \ge 1$ such that $ps > r$. Now, each $A_n:=K^{2n+2p\cdot s}\setminus K^{2n}$, $n > ps$ is covered by at most $N_{2\delta}$ sets of diameter at most $8ps$.
Add $A_{ps}:=K^{2ps}$ to obtain a cover of $G$ uniformly bounded by $8ps$
such that any $r$-ball intersects at most $2\cdot N_{2\delta}$ of its elements. It is so because any $r$-ball can intersect at most two annuli $A_n$.
\end{proof}

\section{Ends of large scale groups}

In this section we apply the general theory of ends of coarse spaces to large scale groups.
\begin{Definition}\label{LocallyBoundedGroupDef}
A large scale group $G$ is \textbf{locally bounded} if for every bounded subset $B$ of $G$ the subgroup $<B>$ of $G$ generated by $B$ is bounded.
\end{Definition}

\begin{Proposition}\label{LocallyBoundedGroupsCase}
Suppose the bornology of a large scale group $G$ has a countable basis.
If $G$ is an unbounded and locally bounded group, then its number of ends is infinite.
\end{Proposition}
\begin{proof}
Express $G$ is a union of a strictly increasing sequence $\{G_i\}_{i\ge 1}$ of its bounded subgroups so that $B_i\subset G_i$, where $\{B_i\}_{i\ge 1}$ is a basis of the bornology of $G$. Choose $g_i\in G_{i+1}\setminus G_i$ for each $i\ge 1$. Given an infinite subset $J$ of natural numbers define $A_J$
as $\bigcup\limits_{i\in J} g_i\cdot G_i$. Notice each $A_J$ is unbounded.
To show $A_J$ is coarsely clopen assume $B\subset G$ is bounded
and choose $j\in J$ so that $B\subset G_j$.
Notice $A_j\subset A_J\cdot G_j\subset A_j\cup\bigcup\limits_{i\in J, i < j} g_i\cdot G_i$, so $(A_J\cdot B)\Delta A_J$ is bounded.
To complete the proof notice $A_J\cap A_L=\emptyset$ if $J\cap L=\emptyset$. Indeed, if $g\in A_J\cap A_L$, then there exist $i < j$
such that $g_i\cdot h=g_j\cdot h'$, where $h\in G_i$ and $h'\in G_j$.
Therefore $g_j\in G_j$, a contradiction. Finally, since $\mathbb{N}$ can be expressed as an infinite union of mutually disjoint infinite subsets of $\mathbb{N}$, $G$ has infinitely many mutually disjoint unbounded coarsely clopen subsets and hence infinitely many ends.

\end{proof}
\begin{Proposition}
Suppose the bornology of a large scale group $G$ has a countable basis. $G$ is non-locally bounded if and only if $G$ admits a strictly increasing sequence $\{G_i\}_{i\ge 1}$ of unbounded boundedly generated subgroups such that $\{G_i\}_{i\ge 1}$ is a basis for boundedly generated subgroups.
\end{Proposition}
\begin{proof}
Let $\{B_i\}_{i\ge 1}$ be a basis of the bornology of $G$, where $\{B_i\}_{i\ge 1}$ is increasing. Choose $g_i\in G\setminus <B_i>$, and set $G_i:=<B_i\cup g_i>$. Notice for any bounded subset $B\subset G$, there exists $i\ge 1$ such that $B\subset B_i$ and hence $<B>\subset G_i$. If each $G_i$ is bounded then $G$ is locally bounded, a contradiction. Therefor, there exists $N\geq 1$ such that $G_n$ is unbounded for all $n\geq N$. Without loss of generality, we may assume that $N=1$.
\end{proof}
\begin{Proposition}\label{SequenceOfSubgroupsProp}
Let $G$ be a large scale group that has a basis $\{G_i\}_{i\ge 1}$ for boundedly generated subgroups consisting of unbounded subgroups. If $A$ is a coarsely clopen unbounded subset of $G$, then there is $n\ge 1$ such that $A\cap G_n$ is unbounded.
\end{Proposition}
\begin{proof}
Let $G_1$ be generated by a symmetric bounded subset $B$ of $G$. Since $(A\cdot B)\Delta A$ is bounded, there is $k > 1$ such that $(A\cdot B)\Delta A\subset G_k$. Therefore $(A\cdot B)\setminus G_k=A\setminus G_k$. If $ A\setminus G_k=\emptyset$, we are done, so assume $g\in A\setminus G_k$. Now, for each $b\in B$, $g\cdot b\in (A\cdot B)\setminus G_k=A\setminus G_k$. That implies $g\cdot G_1\subset A\setminus G_k$. Choose $n>1$ such that $<B\cup g>\subset G_n$, then $g\cdot G_1\subset A\cap G_n$ and hence $A\cap G_n$ is unbounded.
\end{proof}

\begin{Definition}
\textbf{NCC} is a \textbf{shortcut for non-trivial coarsely clopen subsets} $Y$ of a large scale space $X$, i.e. those coarsely clopen subsets that are unbounded and $X\setminus Y$ is unbounded.
\end{Definition}

\begin{Theorem}\label{NumberOfEndsViaFiltration}
Let $G$ be a large scale group that has a basis $\{G_i\}_{i\ge 1}$ for boundedly generated subgroups consisting of unbounded subgroups. If $m\leq \infty$ and each $G_i$ has at most $m$ ends, then the number of ends of $G$ is at most $m$.
\end{Theorem}
\begin{proof}
The case $m=\infty$ is clear, so assume $m < \infty$.

If $G$ has $(m+1)$ mutually disjoint NCC sets $A_i$, $i\leq m+1$, then
by
\ref{SequenceOfSubgroupsProp} we can find an index $n$ such that
each $A_i\cap G_n$ is an NCC set in $G_n$. Hence $|Ends(G_n)|\geq m+1$, a contradiction.
\end{proof}

\begin{Corollary}
\label{NumberOfEndsViaFiltration2}
Let $G$ be a large scale group whose bornology consists of all finite subsets.
If $m\leq \infty$ and $G$ is the union of an increasing sequence $\{G_i\}_{i\ge 1}$ 
of infinite subgroups having at most $m$ ends, then the number of ends of $G$ is at most $m$.
\end{Corollary}
\begin{proof}
Notice $\{G_i\}_{i\ge 1}$ is a basis for boundedly generated subgroups consisting of unbounded subgroups. Apply \ref{NumberOfEndsViaFiltration}.
\end{proof}

\begin{Lemma}\label{ActingOnThreeEnds}
Let $G$ be a large scale group that has a basis $\{G_i\}_{i\ge 1}$ for boundedly generated subgroups consisting of unbounded subgroups that are coarsely geodesic. 
If $G$ contains three NCC sets that are mutually disjoint, then it acts trivially on at most one of the three NCC sets.
\end{Lemma}
\begin{proof}

Suppose $G$ acts trivially on disjoint NCC sets $A_1$, $A_2$ and $A_3$
is an NCC sets disjoint from $A_1\cup A_2$.
Using \ref{SequenceOfSubgroupsProp} we may reduce the proof to $G$ being boundedly generated and of bounded geometry. Equip $G$ with a left-invariant metric $d$ inherited from a Cayley graph.
Find a bounded subset $K$ of $G$ containing $1_G$ such that if $i\leq 3$ and $g\in A_i\setminus K$, $h\in A_i^c\setminus K$, then $d(g,h) > 2$. 

Let $A_4:=G\setminus (A_1\cup A_2\cup A_3)$. Either $A_4$ is an NCC or it is bounded.
Find $m\ge 1$ such that for any $x\in A_i$, $i\leq 3$, of norm at least $m$, $B(x,2\cdot diam(K)+2)$
is contained in $A_i$. If $A_4$ is unbounded, require the same property for $A_4$, otherwise require that $B(x,diam(K)+1)$ is disjoint with $A_4$.

In $A_3$ find an element $g_3$ of the norm bigger than $m$. Hence
$g_3\cdot K\subset A_3$.

Since $A_1\Delta (g_3\cdot A_1)$ is bounded,
choose $g_1\in A_1$ of the norm larger than $m$ such that $g_3\cdot g_1\in A_1$. Given a $K$-chain $c$ joining $g_1$ to $g_0\in K$, it stays in $A_1$ until it hits $K$ for the first time. Truncate $c$ to include initially (until the last element) only elements of $A_1$ and ending at $K$. Now, $g_3\cdot c$ is a $K$-chain starting in $A_1$ and ending in $A_3$. Therefore it hits $K$ at certain moment. That means existence of $x_1\in A_1$ such that $g_3\cdot x_1\in K$.
Similarly, we can find $x_2\in A_2$ such that $g_3\cdot x_2\in K$. That means $g_3^{-1}\cdot K$ intersects both $A_1$ and $A_2$, a contradiction as that set is contained exclusively in only one of $A_i$, $i\leq 4$, due to the norm of
$g_3^{-1}$ being larger than $m$.
\end{proof}

\begin{Theorem}\label{NumberOfEnds}
Let $G$ be a large scale group that has a basis $\{G_i\}_{i\ge 1}$ for boundedly generated subgroups consisting of subgroups that are coarsely geodesic. The number of ends of $G$ is either infinite or at most $2$.
\end{Theorem}
\begin{proof}
Notice for each $G_i$ there exists a symmetric bounded subset $K_i\subset G$ containing $1_G$ such that $G_i=<K_i>$ and that $(K_i^n)_{n\geq 1}$ is a basis for $G_i$. In particular, $\{K_i^n:i,n\in \mathbb{N}\}$ is a countable basis for $G$.
If $G$ is bounded, then $Ends(G)$ is empty.
If $G$ is locally bounded and unbounded, then $Ends(G)$ is infinite by \ref{LocallyBoundedGroupsCase}.

Assume $G$ is unbounded, not locally bounded, its number of ends $m\ge 3$ is finite, and it contains three NCC sets
that are disjoint. Notice that, by \ref{CoarselyEquivalentSpacesHaveHomeomorphicEndsSpaces}, for any $x\in G$, the map $\sigma^x:G\to G$ that maps each $g\in G$ to $x\cdot g$ induces a bijection $\sigma^x_{end}:Ends(G)\to Ends(G)$. Let $Bij(Ends(G))$ be the finite group of all bijections from  $Ends(G)$ to itself. The map $\rho:G\to Bij(Ends(G))$ given by: $\rho(x)= \sigma^x_{end}$ is a group homomorphism with $H:=Ker(\rho)$ is a subgroup of finite index in $G$. Moreover, $H$ acts trivially on $Ends(G)$. Now, there are $m\ge 3$ disjoint NCC subsets of $H$ obtained by intersecting $H$. This contradicts Lemma \ref{ActingOnThreeEnds}.
\end{proof}
\begin{Corollary}
Let $G$ be a locally compact $\sigma$-compact topological group with bornology consisting of pre-compact subsets. The number of ends of $G$ is either infinite or at most $2$.
\end{Corollary}
\begin{proof}
Let $(K_i)_{i\ge 1}$ be an \textbf{exhausting sequence}, i.e. $K_i$ is compact, $K_i\subset int(K_{i+1})$ for each $i\ge 1$, and $\bigcup\limits_{i=1}^\infty K_i=G$. The sequence $\{G_i\}_{i\ge 1}$, where $G_i:=<K_i>$ is a basis for boundedly generated subgroups consisting of subgroups that are coarsely geodesic.
\end{proof}
\begin{Theorem}
Suppose $G$ is a large scale group whose bornology consists of all finite subsets of $G$. If $G$ has finitely many ends, then it has at most $2$ ends.
\end{Theorem}
\begin{proof}
$G$ has a subgroup $H$ of finite index that has the same number of ends as $G$ and $H$ acts on its ends trivially. Suppose $H$ has at least $3$ ends. In that case we can find three mutually disjoint subsets $A_i$, $i\leq 3$, of $H$ which are coarsely clopen and non-trivial. Let $B_i:=A^c_i$, and choose countable unbounded subsets $C_i$ and $D_i$ of each $A_i$ and $B_i$, respectively. Let $H'$ be the subgroup of $H$ generated by $\bigcup\limits_{i=1}^{3}(C_i\cup D_i)$. Each $A_i\cap H'$ is a non-trivial coarsely clopen subset of $H'$ on which $H'$ acts trivially contradicting \ref{ActingOnThreeEnds}.
\end{proof}

Now we generalize the Stallings' theorem by showing that any group of bounded geometry of two ends contains an infinite cyclic subgroup of bounded index.

\begin{Lemma}\label{CyclicSubgroupsLemma}
Let $G$ be a large scale group containing two unbounded cyclic subgroups $H$ and $K$. If $H$ is of bounded index in $G$, then so is $K$. 
\end{Lemma}
\begin{proof}
Since there is a coarse equivalence $f:G\to H$ inverse to the inclusion $i:H\to G$, $f|K:K\to H$ is a coarse embedding, hence a coarse equivalence. Consequently, the inclusion $K\to G$ is a coarse equivalence and $K$ is of bounded index in $G$. 
\end{proof}

\begin{Theorem}\label{SpecialStallingsTwoEnds}
If $G$ is a large scale group of bounded geometry, non-locally bounded, and $\sigma$-bounded that has two ends, then it contains an infinite cyclic subgroup of bounded index.
\end{Theorem}
\begin{proof}
Notice $G$ contains a subgroup $H$ of finite index that acts trivially on the ends of $G$, so $H$ has two ends and acts trivially on them. Thus, we reduce the general case to that of $G$ acting trivially on its ends. Choose a coarsely clopen subset $A$ belonging to one of ends of $G$.

First, consider the case of $G$ being boundedly generated. 
Using \ref{FinitelyManyComponentsBdGrowth} we may find a bounded subset $B_1$ of $G$ containing $1_G$ so that
each $A\setminus B_1$ and $A^c\setminus B_1$ are unions of unbounded $1$-components of $G\setminus B_1$ and $A\setminus B_1$ is $2$-separated from $A^c\setminus B_1$ for some Cayley metric $d$ on $G$. Therefore $A\setminus B_1$ and $A^c\setminus B_1$ are $1$-components of $G\setminus B_1$ and
$G$ acts trivially on each of those components. Now, \ref{CyclicSubgroupProp} says $G$ has a cyclic subgroup of bounded index.

Suppose $G$ is not boundedly generated and is the union of its bounded subsets $B_i$, $i\ge 1$. We will show that there exists a strictly increasing sequence $H_n$ of subgroups of $G$ satisfying the following conditions:\\
1. $H_1$ is infinite cyclic,\\
2. $H_n$ is of bounded index in $H_{n+1}$ for each $n\ge 1$,\\
3. $G$ is the union of all $H_n$, $n\ge 1$.

Given a bounded subset $B$ of $G$ we can find using \ref{SequenceOfSubgroupsProp} a boundedly generated subgroup $H_B$ of $G$
containing $B$ such that both $A\cap H_B$ and $A^c\cap H_B$ are NCC sets in $H_B$.
By the first case, $H_B$ has an infinite cyclic subgroup of bounded index. Call this group $H_1$.

In particular, if we construct two subgroups
$H_B\subset H_{B'}$ that way, then $H_B$ is of bounded index in $H_{B'}$. To this end, notice that $H_{B'}$ has an infinite cyclic subgroup $H'_{1}$ of bounded index. Since $H_1$ and $H'_{1}$ are infinite cyclic subgroups of $H_{B'}$ and $H'_{1}$ of bounded index in $H_{B'}$, by \ref{CyclicSubgroupsLemma}, $H_1$ must be of bounded index in  $H_{B'}$ and hence $H_B$ is of bounded index in $H_{B'}$. Using these facts it is easy to construct the required sequence $H_n$ of subgroups of $G$.

 $A$ cannot be contained in any $H_m$. Indeed, suppose there is $m\ge 1$ such that
$A\subset H_m$ and choose $z\in H_{m+1}\setminus  H_{m}$. On one hand, if $A\cap A\cdot z\neq\emptyset$, then $z\in A^{-1}\cdot A\subset H_m$, a contradiction. On the other hand, if $A\cap A\cdot z=\emptyset$, then $A\Delta (A\cdot z)=A\cup A\cdot z$ which is unbounded, a contradiction.

$A\cap H_1$ and $A^c\cap H_1$ are both unbounded. Let $t$ be a generator of $H_1$. Since both $A\Delta(A\cdot t)$ and $A^c\Delta(A^c\cdot t)$ are bounded,
there is $k > m$ such that both these sets are contained in $H_k$.
Given $x\in A\cap H_{k+1}\setminus H_k$, then $x\cdot t\notin H_k$. One has $x\cdot t\in A$ as otherwise $x\cdot t\in A\Delta(A\cdot t)\subset H_k$. Consequently, $x\cdot t^n\in A$ for all integers $n$.
As $G$ acts trivially on $A$, $x^{-1}\cdot A\setminus A$ is bounded.
That implies $H_1\cap A^c$ is bounded, a contradiction.
\end{proof}
\begin{Corollary}
Let $G$ be a locally compact $\sigma$-compact topological group with bornology consisting of pre-compact subsets. If $G$ is not compactly generated, then either $G$ is 1-ended or it has infinitely many ends.
\end{Corollary}

\begin{Theorem}\label{GeneralStallingsTwoEnds}
If $G$ is a large scale group of bounded geometry with $2$ ends, then the following conditions are equivalent:\\
1. $G$ is boundedly generated of bounded geometry and has $2$ ends.\\
2. $G$ is boundedly generated of bounded geometry, $\sigma$-bounded that has $2$ ends.\\ 
3. $G$ contains an infinite cyclic subgroup of bounded index.
\end{Theorem}
\begin{proof}
1)$\implies$2) is obvious.\\
2)$\implies$3) follows from
\ref{SpecialStallingsTwoEnds}.\\
3)$\implies$1) is obvious.\\
\end{proof}

\begin{Theorem}
Let $G$ be a coarse group whose bornology consists of all finite subsets.  If $G$ has $2$ ends, then $G$ is finitely generated. Therefore it contains an infinite cyclic subgroup of finite index.
\end{Theorem}
\begin{proof}
$G$ has a subgroup $G_1$ of finite index that has $2$ ends and $G_1$ acts on its ends trivially. If $G_1$ is countable, then we are done, so assume it is uncountable. 
Given a countable subgroup $G_2$ of $G_1$ choose $A_1\subset G_1$, a non-trivial coarsely clopen subset on which $G_1$ acts trivially. Let $A_2:=G_1\setminus A_1$ and choose infinite countable subsets $C_i$ of $A_i$, $i\leq 2$. The group generated by $G_2\cup C_1\cup C_2$ has two ends since the countable group $<G_2\cup C_1\cup C_2>$ acts trivially on $A_1\cap <G_2\cup C_1\cup C_2>$ and $A_2\cap <G_2\cup C_1\cup C_2>$, by \ref{GeneralStallingsTwoEnds}, it is finitely generated. That means every countable subgroup $G_2$ of $G_1$ is contained in a finitely generated subgroup $G_3$ of two ends. By induction we can construct a strictly increasing sequence $H_i$ of finitely generated subgroups of $G_1$ each having $2$ ends inherited from $G_1$.
The union of all $H_i$ is not finitely generated but has $2$ ends, a contradiction to \ref{GeneralStallingsTwoEnds}.
\end{proof}

\end{document}